\RequirePackage{fix-cm}
\documentclass[smallcondensed]{svjour3hack}     
\smartqed  

\journalname{Journal of Global Optimization}

\usepackage{graphicx}

\usepackage{amsmath,amsthm,amssymb}
\usepackage{xcolor}
\usepackage{float}
\usepackage{multicol}
\usepackage{placeins}

\usepackage[utf8]{inputenc}
\usepackage[hypertexnames=false,colorlinks=true,breaklinks=true,bookmarks=true,urlcolor=blue,citecolor=blue,linkcolor=blue,bookmarksopen=false,draft=false]{hyperref}

\usepackage[T3,T1]{fontenc}
\DeclareSymbolFont{tipa}{T3}{cmr}{m}{n}
\DeclareMathAccent{\invbreve}{\mathalpha}{tipa}{16}
\usepackage{bm}

\allowdisplaybreaks[1]

\DeclareMathOperator{\conv}{conv}
\DeclareMathOperator{\vol}{vol}
\DeclareMathOperator{\Diag}{Diag}
\DeclareMathOperator{\cl}{cl}

\makeatletter
\let\c@proposition\c@theorem
\let\c@corollary\c@theorem
\let\c@lemma\c@theorem
\let\c@definition\c@theorem
\let\c@example\c@theorem
\makeatother

\begin{document}

\title{Gaining or losing perspective\footnote{This paper is an extended version of the paper by the same name appearing in the proceedings WCGO 2019 (Metz, France).}\thanks{J. Lee was supported in part by ONR grant N00014-17-1-2296 and LIX, l'\'Ecole Polytechnique. D. Skipper was supported in part by ONR grant N00014-18-W-X00709.
E. Speakman was supported in part by the Deutsche Forschungsgemeinschaft (DFG, German Research Foundation) - 314838170, GRK 2297 MathCoRe.}
}%

\titlerunning{Gaining or losing perspective}

\author{Jon Lee \and
Daphne Skipper \and \\
Emily Speakman}

\authorrunning{J. Lee,
D. Skipper \&
E. Speakman}

\institute{Jon Lee \at IOE Dept., University of Michigan. Ann Arbor, Michigan, USA\\  \email{jonxlee@umich.edu} \and Daphne Skipper \at Department of Mathematics, U.S. Naval Academy. Annapolis, Maryland, USA\\
\email{skipper@usna.edu} \and Emily Speakman \at Department of Mathematical and Statistical Sciences, University of Colorado Denver, Denver, Colorado, USA
\email{emily.speakman@ucdenver.edu}}

\date{\today}%

%


\maketitle
\begin{abstract}
We study MINLO (mixed-integer nonlinear optimization) formulations of the disjunction
$x\in\{0\}\cup[l,u]$, where $z$ is a binary indicator
of $x\in[l,u]$ ($u> \ell > 0$), and $y$ ``captures'' $f(x)$, which is
assumed to be convex on its domain
$[l,u]$, but otherwise $y=0$ when $x=0$.
This model is useful when  activities have operating ranges,
we pay a fixed cost for carrying out each activity, and
costs on the levels of activities are  convex.

Using volume as a measure to compare convex bodies,
we investigate  a variety of continuous relaxations of this model, one of which is
 the convex-hull, achieved via the ``perspective reformulation'' inequality
$y \geq zf(x/z)$. We compare this to various weaker relaxations, studying when they
may be considered as viable alternatives.
 In the important special case when $f(x) := x^p$, for $p>1$,   relaxations utilizing the  inequality $yz^q \geq x^p$, for $q \in [0,p-1]$, are higher-dimensional power-cone representable, and hence tractable in theory.
One well-known concrete application (with $f(x) := x^2$) is mean-variance optimization (in the style of Markowitz), and we carry out some experiments to illustrate our theory on this
application.
\end{abstract}

\keywords{mixed-integer nonlinear optimization \and volume  \and integer \and relaxation \and polytope \and perspective \and higher-dimensional power cone \and exponential cone}


\section*{Introduction}

\subsubsection*{Background.}
Our interest is in studying ``perspective reformulations'' and alternatives
for a specific situation involving
indicator variables:
 when an indicator is ``off'',  a vector of
decision variables is forced to a specific point, and when it is ``on'',
the vector of decision variables must belong to a specific convex set.
\cite{gunlind1}  studied such a situation where binary variables
manage terms in a separable-quadratic  objective function,
with each continuous variable $x$ being either 0 or in a positive interval (also see \cite{Frangioni2006}).
More generally, we are interested in separable objectives with convex terms.  In the special case when terms
have the form $x^p$, $p>1$, the perspective-reformulation approach (see \cite{gunlind1} and the references therein)
leads to very strong conic-programming relaxations, but not all MINLO (mixed-integer nonlinear optimization)
solvers are  equipped to handle these. So one of our interests is in
 determining when a natural and simpler non-conic-programming
relaxation may be adequate.

Generally, our view is that MINLO modelers and algorithm/software
developers can usefully factor in analytic comparisons of
relaxations in their work.
$d$-dimensional volume is a natural analytic measure for comparing
the size of a pair of  convex bodies in $\mathbb{R}^d$.
 \cite{LM1994} introduced the idea of
 using volume as a measure for comparing relaxations (for fixed-charge, vertex packing, and other
 relaxations).
  \cite{SpeakmanLee2015,SpeakmanLee_Branching,SpeakmanYuLee,SpeakmanThesis,SpeakAkerov2019}
used the idea to compare convex relaxations of graphs of trilinear monomials on
 box domains.
\cite{KLS1997,Stein,LeeSkipperBQP2017}
 compared relaxations of graphical Boolean-quadric polytopes. \cite{BCSZ} and \cite{DM}
 used ``volume cut off'' as a measure for the strength of cuts.

The current relevant convex-MINLO software environment
 is very unsettled with a lot to come.
 One of the best algorithmic options for convex-MINLO is ``outer approximation'',
but this is not usually appropriate when constraint functions are not convex
(even when the feasible region of the continuous relaxation is a convex set).
Even ``NLP-based B\&B'' for convex-MINLO
may not be appropriate when the underlying NLP solver
is presented with a formulation where a constraint qualification does not hold
at likely optima. In some situations, the relevant convex sets
can be represented as convex cones, thus handling the constraint-qualification issue,
but then limiting the choice of solvers to ones that are equipped to work with cone constraints.
In particular, conic constraints are not well handled by
 general convex-MINLO software (like Knitro, Ipopt, Bonmin, etc.).
 As of the present moment, the only conic solver that handles
 integer variables (via B\&B) is MOSEK.
 But even MOSEK is not equipped to handle all possible cones,
 and not all cones are handled very efficiently, especially when
 we access the solver through a modeling framework like CVX.
 So not all of our work can be applied
 today, within the current convex-MINLO software environment, and so we see our work as
 forward looking.

\subsubsection*{Our contribution and organization.}

We study MINLO (mixed-integer nonlinear optimization) formulations of the disjunction
$x\in\{0\}\cup[l,u]$, where $z$ is a binary indicator
of $x\in[l,u]$ ($u> \ell > 0$), and $y$ ``captures'' $f(x)$, which is assumed to be
convex on its domain
$[l,u]$, but otherwise $y=0$ when $x=0$.
We study various models and relaxations for this situation,
in particular, the ``perspective relaxation''.
We also consider a simpler ``na\"{\i}ve relaxation''. For this, we need to assume that:
the domain of $f$ is all of $[0,u]$,
$f$ is convex on $[0,u]$, $f(0)=0$,
and $f$ is  increasing on $[0,u]$.
Additionally, we consider the effect of first tightening such functions
on $[0,\ell]$.
To go even further, we look at the very important case of
$f(x):=x^p$, with $p>1$.
In this situation, we investigate a family of relaxations for this model,
``interpolating'' between the perspective relaxation and the na\"{\i}ve relaxation.


In \S\ref{sec:def}, we formally define the sets relevant to our study.
In \S\ref{sec:volper}, we derive a general formula for the volume of the perspective relaxation.
In \S\ref{sec:naive}, we derive a general formula for the volume of the na\"{\i}ve  relaxation.
Armed with these formulae, we are in a position to
quantify, in terms of $f$, $\ell$ and $u$, how much stronger the
perspective relaxation is compared to the  na\"{\i}ve  relaxation.
Also, we apply the na\"{\i}ve  relaxation
after first tightening the base function on $[0,\ell)$ (outside of its defined domain),
 and we study the
effect in some detail for convex power function $f(x):=x^p$, with $p>1$.
In \S\ref{sec:power}, we look more closely at convex power functions, and we study  relaxations
``interpolating'', using a parameter $q\in[0,p-1]$ (the ``lifting exponent''), between the perspective relaxation and the na\"{\i}ve relaxation.
$q=0$ corresponds to the na\"{\i}ve relaxation, and $q=p-1$ corresponds to the perspective relaxation.
In doing so, we
quantify, in terms of $\ell,u,p$, and $q$, how much stronger the
perspective relaxation is compared to the weaker relaxations, and when, in terms of $\ell$ and $u$, there
is much to be gained at all by considering more than the weakest relaxation. Using our
volume formula,
and thinking of the baseline of $q=1$, which we dub the ``na\"{\i}ve perspective relaxation'',
we quantify the impact of ``losing perspective'' (e.g., going to $q=0$,
namely the  na\"{\i}ve relaxation) and of ``gaining perspective'' (e.g., going to $q=p-1$,
namely the  convex hull).
 In \S\ref{sec:comp}, we present some computational experiments  which bear out our theory, as we verify
that volume can be used to determine which variables are more
important to handle by perspective relaxation. Depending on $\ell$ and $u$ for a particular $x$-variable (of which there
may be a great many in a real model), we may adopt different relaxations based on the differences of the
volumes of the various relaxation choices and on
the solver environment.
In \S\ref{sec:conc}, we make some brief concluding remarks.

Compared to earlier work on volume formulae
relevant to comparing convex relaxations, our present results are the first involving convex sets that are not polytopes. Thus we demonstrate that
we can get meaningful results that do not rely implicitly or explicitly  on triangulation of polytopes.



\subsubsection*{Notation and simple but useful facts.} Throughout, we use boldface lower-case for vectors and boldface upper-case for matrices, vectors are column vectors,
$\| \cdot \|$ indicates the 2-norm, and for
a vector $\mathbf{x}$, its transpose is indicated by $\mathbf{x}'$~.

We make free use of the following simple lemma.

\begin{lemma}[``three-secant inequality'' (see \cite{PClark}, for example)]\label{lem:3sec}
Suppose that $f:\mathbb{R}\rightarrow\mathbb{R}$ is convex on the interval $I$. Then
for all $a<x<b$ in $I$, we have
\[
\frac{f(x)-f(a)}{x-a} \leq
\frac{f(b)-f(a)}{b-a} \leq
\frac{f(b)-f(x)}{b-x}.
\]
\end{lemma}

%


\section{Our sets}\label{sec:def}

\subsection{Definitions.}
For real scalars
 $u> \ell > 0$ and univariate convex function $f$, we define
 \begin{align*}
\invbreve{D}_f := \invbreve{D}_f(\ell,u):=
&\conv\left(
\{(0,0,0)\} \bigcup
\left\{ (x,y,1) \in \mathbb{R}^3 ~:~ 
\vphantom{\frac{f(u)-f(\ell)}{u-\ell}}
\right.\right.\\
&\left.\left. f(\ell) + \frac{f(u)-f(\ell)}{u-\ell} (x-\ell)
\geq y \geq f(x),~ u\geq x \geq \ell,
\right\}
\right).
\end{align*}
This ``disjunctive set'' captures that we want either $x=y=z=0$
or $z=1$, $y$ upper bounding $f$, and $y$ not above the
secant of the curve of $f(x)$ between $x=\ell$ and $x=u$.
The secant condition is introduced from a practical point of view
--- in the context of convex modeling,
%
%
we can
constrain $y$ from above by any concave
function of $x$ that is an \emph{upper} bound on $f$
for $x\in[\ell,u]$. Doing this in the
tightest possible manner, by using the secant, leads us to $\invbreve{D}_f$.

We are interested in relaxations of this set related to ``perspective transformation''.
For a convex function $h$, the \emph{perspective} of $h$ is
the convex function
\[
\tilde{h}(x,z):= \left\{
           \begin{array}{ll}
             z h(x/z), & \hbox{for $z>0$;} \\
             +\infty, & \hbox{otherwise.}
           \end{array}
         \right.
\]
Importantly, if we evaluate the closure of $\tilde{h}$
at $(0,0)$, we get $0$. See \cite{perspecbook}
for more information on perspective functions.
This transformation leads to the \emph{perspective relaxation}
\begin{align*}
&\invbreve{S}^*_f := \invbreve{S}^*_f(\ell,u) :=\\
&\cl \left\{ (x,y,z) \in \mathbb{R}^3 ~:~
\left(f(\ell)-  \frac{f(u)-f(\ell)}{u-\ell} \ell\right)z + \frac{f(u)-f(\ell)}{u-\ell} x
\geq y \geq z f(x/z),~ \right.\\
&\left. uz\geq x
\geq \ell z,~ 1\geq z > 0,~ y\geq 0
\vphantom{\frac{f(u)-f(\ell)}{u-\ell}}
\right\},
\end{align*}
where $\cl$ denote the closure operator. 
Intersecting $\invbreve{S}^*_f$ with the hyperplane defined by $z=0$, leaves the
single point $(x,y,z)=(0,0,0)$. In this way, the ``perspective and closure'' construction
gives us exactly the value $y=0$ that we want at $x=0$.


It is interesting to compare the
inequalities
\begin{equation}\label{ybounds}
f(\ell) + \frac{f(u)-f(\ell)}{u-\ell} (x-\ell)
\geq y \geq f(x)
\end{equation}
(this is just the inequalities coming from the definition of the disjunctive set
$\invbreve{D}_f$)
with the
inequalities
\begin{equation}\label{yperspecbounds}
\left(f(\ell)-  \frac{f(u)-f(\ell)}{u-\ell} \ell\right)z + \frac{f(u)-f(\ell)}{u-\ell} x
\geq y \geq z f(x/z)
\end{equation}
coming from the definition of $\invbreve{S}^*_f$.
First, is easy to see that \eqref{yperspecbounds} reduces to
\eqref{ybounds} when we set $z=1$.
Obviously the right-hand inequality of \eqref{yperspecbounds}
is the perspective of the right-hand inequality of \eqref{ybounds},
and because the perspective transformation preserves convexity,
we still have a convex function lower bounding $y$.
Moreover, the
left-hand inequality of \eqref{yperspecbounds}
is the perspective of the left-hand inequality of \eqref{ybounds},
and because the perspective  transformation preserves \emph{linearity},
we still have a concave (indeed linear) function upper bounding $y$.

Finally, we consider the special situation in which:
the domain of $f$ is all of $[0,u]$,
$f$ is convex on $[0,u]$, $f(0)=0$,
and $f$ is  increasing on $[0,u]$.
For example, $f(x):=x^p$ with $p>1$ has these properties.
Here, we define the \emph{na\"{\i}ve relaxation}
\begin{align*}
&\invbreve{S}^0_f:=\invbreve{S}^0_f(\ell,u):=\\
&\left\{ (x,y,z) \in \mathbb{R}^3 ~:~
\left(f(\ell)-  \frac{f(u)-f(\ell)}{u-\ell} \ell\right)z
  + \frac{f(u)-f(\ell)}{u-\ell} x
\geq y \geq f(x),~ \right.\\
&\left. uz\geq x \geq  \ell z,~   1\geq z \geq 0,
\vphantom{\frac{f(u)-f(\ell)}{u-\ell}}
\right\}.
\end{align*}
Note how in this model we have not taken the perspective of the nonlinear function
$f$. Because of this, we really need that the domain of $f$ is all of $[0,u]$.
But we have in effect used the tightest  upper bound
for $f$ on $[\ell,u]$,
that is linear (even concave) in $x$ and $z$, to
constrain $y$   from above.

Briefly summarizing our notation involving $S$,
$~\invbreve{\cdot}~$ (``cap'' ) indicates presence of a very particular upper bound on $y$,
no superscript indicates that $z\in\{0,1\}$, superscript 0 means relaxation to $1\geq z \geq 0$,
superscript $*$ means perspective relaxation.

Our goal is to analyze and compare, via volume, the various convex relaxations of
$\invbreve{S}_f(\ell,u)$, studying the dependence on $f$, $\ell$ and $u$.
We have that $\invbreve{S}_f  \subseteq \conv(\invbreve{S}_f) =\invbreve{S}^*_f \subseteq \invbreve{S}^0_f$.
In particular, we will focus on comparing the tighter $\invbreve{S}^*_f$ with the computationally more-tractable
but looser $\invbreve{S}^0_f$.

\subsection{Numerical difficulties and not.}

Before continuing with our main development, we wish to look a bit at our relaxations
from the point of view of numerical reliability. First, we consider the
na\"{\i}ve relaxation $\invbreve{S}^0_f$. Recall that in this case,
we assume that the domain of $f$ is all of $[0,u]$,
$f$ is convex on $[0,u]$, $f(0)=0$,
and $f$ is increasing on $[0,u]$.
For convenience, and because we are thinking about the application of
solvers that require smoothness for convergence, we will assume that $f$ is differentiable at 0.
The most troublesome point is $(x,y,z)=(0,0,0)$, where all of the defining inequalities are satisfied as equations, except for $1\geq z$.
We wish to show that the MFCQ (Mangasarian-Fromowitz constraint qualification; see \cite{Bertsekas}, for example) is satisfied at this point.

Choose $\alpha$ so that $1+\ell/u < \alpha < 2$, and for simplicity of notation, let $m := \frac{f(u)-f(\ell)}{u-\ell}$.  Consider the direction
\[ (d_x, d_y, d_z) := \left(1, ~m-\frac{u+\ell}{\alpha u}\left(m - \frac{f(\ell)}{\ell} \right),~ \frac{u+\ell}{2u \ell}\right). \]

Considering the constraints $uz\geq x \geq \ell z$, what we want is
\[
u d_z > 1 > \ell d_z,
\]
which is satisfied by our choice of $d_z$.

Considering the constraint $f(x)-y\leq 0$, we need $f'(0) - d_y < 0$.  We have,
\[
f'(0)-d_y = -(m - f'(0))+ \frac{u+\ell}{\alpha u} \left(m - \frac{f(\ell)}{\ell} \right),
\]
which is negative because $m > f(\ell)/\ell > f'(0)$ by convexity, and because
\[
1 =  \frac{u+\ell}{(1 +\ell/u)u} > \frac{u+\ell}{\alpha u}.
\]

Finally, for the constraint
\[
y - \left(f(\ell)-  \frac{f(u)-f(\ell)}{u-\ell} \ell\right)z
  - \frac{f(u)-f(\ell)}{u-\ell} x \leq 0,
\]
the dot product of the gradient of the constraint with the direction simplifies to
\[
\frac{u+\ell}{u}\left(m - \frac{f(\ell)}{\ell} \right)\left(\frac{1}{2} - \frac{1}{\alpha}\right).
\]
The first factor is obviously positive, the second factor is positive by convexity (using Lemma \ref{lem:3sec}), and the last factor is negative because $0 < \alpha < 2$.


Therefore, we can conclude that MFCQ
holds at $(x,y,z)=(0,0,0)$. So we can reasonably hope that
NLP solvers will not have trouble with the na\"{\i}ve relaxation $\invbreve{S}^0_f$.

Considering, instead, the perspective relaxation $\invbreve{S}^*_f$, we can see \emph{potential}
trouble.
For example, for $f(x):=x^p$, $p>1$, the inequality
$y \geq z f(x/z)$ becomes
$x^p - yz \leq 0$. The gradient of this latter constraint is $(px^{p-1}, -z, -y)$,
and clearly then, MFCQ cannot hold at $(x,y,z)=(0,0,0)$.
This plainly suggests that we have to be careful in working with
the perspective relaxation $\invbreve{S}^*_f$, in the context of generic NLP solvers.
In fact, we can get around this issue, in many practical circumstances, using conic solvers.
We will return to this issue, later, as we examine specific functions $f$ in detail.

\section{The volume for the perspective relaxation}
\label{sec:volper}

In this section, we derive a formula for the volume of $\invbreve{S}^*_f(\ell,u)$ by noticing that $\invbreve{S}^*_f(\ell,u)$ is a pyramid in $\mathbb{R}^3$.



\begin{theorem}  \label{thm.perpective_general_f}
Suppose that $f$ is a nonnegative, continuous, and convex function on $[\ell, u]$, for $u > \ell > 0$.  Then
\begin{equation}\label{eq.perspective_general_f}
\vol(\invbreve{S}^*_f(\ell,u)) = \frac{1}{6}(u-\ell)(f(u)+f(\ell)) - \frac{1}{3}\int_{\ell}^u f(x) dx~.
\end{equation}
\end{theorem}

\begin{proof}
Thinking of $\invbreve{S}^*_f(\ell,u)$ as the convex hull of $\invbreve{S}_f(\ell,u)$, we note that $\invbreve{S}^*_f(\ell,u)$ is the pyramid in $\mathbb{R}^3$ with
 apex $(0,0,0)$ and base a 2-dimensional convex set in the plane $z=1$ defined by the system of inequalities,
\begin{align*}
      f(x) ~\leq~&y~ \leq ~ f(\ell) + \frac{f(u) - f(\ell)}{u-\ell}(x - \ell) \\
	\ell ~\leq~&x~ \leq ~ u.
\end{align*}

It is well known that the volume of a such a pyramid is $\frac{1}{3}BH$, where $B$ is the area of the base, and $H$ is the perpendicular height of the pyramid.  In this case, the height is the distance between the parallel planes that the vertex and the base live in, those described by $z=0$ and $z=1$, respectively; so $H = 1$.  All that is left is to calculate the area of the base via the integral,
\begin{align*}
B &= \int_{\ell}^u \left[f(\ell) + \frac{f(u) - f(\ell)}{u-\ell}(x - \ell) - f(x)\right] dx \\
   &= \frac{1}{2}(u-\ell)(f(u)+f(\ell)) - \int_{\ell}^u f(x)dx.
\end{align*}
\end{proof}

In the following result, we apply Theorem  \ref{thm.perpective_general_f} to a
general increasing exponential function. 
We return to a particular case of such a function
 in Section \ref{sec:naive}, where we compare the volume of $\invbreve{S}^*_{f}(\ell,u)$ to the volume of another natural relaxation of $S_{f}(\ell,u)$. 

\begin{corollary} \label{ex:expper}
Let $f(x) := b^x+a$,  with $b>1$ and $a\in\mathbb{R}$, and domain
$[l,u]$, with $u> \ell > 0$.
  Then
\begin{equation*}
\vol(\invbreve{S}^*_{f}(\ell,u)) = \frac{1}{6}(u-\ell)(b^u+b^{\ell})- \frac{1}{3\ln b}(b^u-b^{\ell}).
\end{equation*}
\end{corollary}
\medskip

The perspective of the function $f(x) := b^x+a$ is handled by MOSEK using the
 ``3-dimensional exponential cone'' (see \cite{cookbook}, Chapter 5), the closure of the points in $\mathbb{R}^3$ satisfying $x_1\geq x_2 \exp(x_3/x_2)$, with $x_1,x_2>0$.
 In this case, the inequality
 $y\geq z f(x/z)$ becomes
 \[
 \underbrace{y-a}_{x_1} \geq \underbrace{z}_{x_2} \hbox{exp}\biggl(\frac{\overbrace{\log(b)x}^{x_3}}{\underbrace{z}_{x_2}}\bigg),
 \]
 which is in the format of an exponential cone constraint.

\section{The volume for the na\"{\i}ve relaxation}
\label{sec:naive}

Recall the definition of the na\"{\i}ve relaxation, i.e.,

\begin{align*}
&\invbreve{S}^0_f(\ell,u):=\\
&\left\{ (x,y,z) \in \mathbb{R}^3 ~:~
\left(f(\ell)-  \frac{f(u)-f(\ell)}{u-\ell} \ell\right)z
  + \frac{f(u)-f(\ell)}{u-\ell} x
\geq y \geq f(x),~ \right.\\
&\left. uz\geq x \geq \ell z,~   1\geq z \geq 0
\vphantom{\frac{f(u)-f(\ell)}{u-\ell}}
\right\}.
\end{align*}

\bigskip
\noindent Note that  for $\invbreve{S}^0_f(\ell,u)$ to be well defined, $f$ must be defined on all of $[0,u]$. To force $y=0$ when $x=z = 0$, we must have $f(0) = 0$. In the next subsection,
we assume that $f(0) = 0$, and we compute the volume of $\invbreve{S}^0_f(\ell,u)$.
After that, we describe a related way to deal with the case of $f(\ell)>0$.
This last case can be particularly relevant when $f$ is not defined on $[0,\ell)$.
But,
as we will see, it can even be relevant when $f$ has a natural definition on
all of $[0,u])$.


\subsection{$f(0)=0$.}

In order to compute the volume of $\invbreve{S}^0_f(\ell,u)$, we first introduce a second valid (but simpler) upper bound on the variable $y$. In applications, $y$ is meant to model/capture $f(x)$ via the minimization pressure
of an objective function. So we introduce the linear inequality $zf(u) \geq y$, which captures that $z = 0$ implies $f(x) = 0$, and that $z = 1$ implies $f(u) \geq f(x)$.  We define the following relaxation which replaces our original upper bound with this simpler one.  Here, the $\bar{\cdot}$ (as opposed to $\invbreve{\cdot}$), denotes that the simpler bound is being used.

\begin{align*}
\bar{S}^0_f(\ell,u):=\left\{ (x,y,z) \in \mathbb{R}^3 ~:~
zf(u) \geq y \geq f(x),~ uz\geq x \geq \ell z,~   1\geq z \geq 0
\vphantom{\frac{f(u)-f(\ell)}{u-\ell}}
\right\}.
\end{align*}

Note that we can also define the perspective relaxation with the simpler upper bound on $y$.

\begin{align*}
\bar{S}^*_f(\ell,u):=\cl\left\{ (x,y,z) \in \mathbb{R}^3 ~:~
zf(u) \geq y \geq zf(x/z),~ uz\geq x \geq \ell z,~   1\geq z \geq 0
\vphantom{\frac{f(u)-f(\ell)}{u-\ell}}
\right\}.
\end{align*}

Now we define a simplex, $\Delta_f(\ell,u)$, as follows.

\begin{align*}
\Delta_f(\ell,u) := \conv\left\{(0,0,0),(\ell,f(\ell),1),(\ell,f(u),1),(u,f(u),1)\right\}.
\end{align*}

As the following lemma demonstrates, $\Delta_f(\ell,u)$ is exactly what a set ``gains'' when we use this simpler bound on $y$.  Therefore, if we wish to obtain the volume of $\invbreve{S}^0_f(\ell,u)$, for example, we can compute the volume of $\bar{S}^0_f(\ell,u)$ (a simpler task) and subtract the volume of the simplex.

\begin{lemma} \label{lem.delta}
Suppose that  with $u> \ell > 0$.
If $f$ is continuous, convex, and nonnegative on $[\ell,u]$,
then
\begin{align} \label{pdelta} \cl\left\{ \bar{S}^*_f(\ell,u) \setminus \invbreve{S}^*_f(\ell,u)\right\} = \Delta_f(\ell,u).
\end{align}
If $f$ is continuous, convex, and increasing on $[0,u]$, with $f(0) = 0$, then
\begin{align} \label{ndelta} \cl\left\{\bar{S}^0_f(\ell,u) \setminus \invbreve{S}^0_f(\ell,u)\right\} =  \Delta_f(\ell,u).
\end{align}
Moreover, $$\vol(\Delta_f(\ell,u)) = \frac{1}{6}(f(u)-f(\ell))(u-\ell).$$
\end{lemma}

\begin{proof}
Under either set of hypotheses on $f$, the simplex $\Delta_f(\ell,u)$ has facet defining inequalities,
\begin{align}
y &\geq \left(f(\ell)-  \frac{f(u)-f(\ell)}{u-\ell} \ell \right)z + \frac{f(u)-f(\ell)}{u-\ell} x;  \label{f1} \\
y &\leq zf(u);   \label{f2}  \\
x &\geq z\ell;   \label{f3} \\
z &\leq 1.  \label{f4}
\end{align}
To see this, note that $(0,0,0)$ lies on (i.e., satisfies the related inequality with equality) facets defined by \eqref{f1}, \eqref{f2}, and \eqref{f3}, $(\ell,f(\ell),1)$ lies on facets defined by \eqref{f1}, \eqref{f3}, and \eqref{f4}, $(\ell,f(u),1)$ lies on facets defined by \eqref{f2}, \eqref{f3}, and \eqref{f4}, and $(u,f(u),1)$ lies on facets defined by \eqref{f1}, \eqref{f2}, and \eqref{f4}.

Moreover, it is easy to check that $x \leq uz$ and $z \geq 0$ are satisfied by all four extreme points of $\Delta_f(\ell,u)$.  Combining all of these inequalities, we have,
\[
\Delta_f(\ell,u) =\left\{ (x,y,z) \in \mathbb{R}^3 ~:~
zf(u) \geq y \geq L(x,z),~ uz\geq x \geq \ell z,~   1\geq z \geq 0
\vphantom{\frac{f(u)-f(\ell)}{u-\ell}}
\right\},
\]
where $L(x,z) := \left(f(\ell) -  \frac{f(u)-f(\ell)}{u-\ell} \ell \right)z + \frac{f(u)-f(\ell)}{u-\ell} x.$  Therefore, $\Delta_f(\ell,u)$ differs from $\bar{S}^0_f(\ell,u)$ and $\bar{S}^*_f(\ell,u)$ only in their lower bounds on $y$.

To complete the proof of the first two statements, it suffices demonstrate that $y \geq L(x,z)$ dominates both $y \geq f(x)$ (under the hypotheses for equation (\ref{pdelta})) and $y \geq zf(x/z)$ (under the requirements for equation (\ref{ndelta})) for $x \in [\ell z, u z]$ and $z \in (0,1)$.  This will imply that $\Delta_f(\ell,u)$ is formed by adding the inequality $y \geq L(x,z)$ to $\bar{S}^0_f(\ell,u)$ (or $\bar{S}^*_f(\ell,u)$), whereas $\invbreve{S}^0_f(\ell,u)$ ($\invbreve{S}^*_f(\ell,u)$) is formed by adding the inequality $y \leq L(x,z)$ to $\bar{S}^0_f(\ell,u)$ ($\bar{S}^*_f(\ell,u)$).

Fix $\hat{z} \in (0,1)$.

Suppose that $f$ satisfies the requirements for equation (\ref{ndelta}).  Because $f$ is convex on $[\ell \hat{z}, u \hat{z}] \subseteq [0,u]$ and $L(x,\hat{z})$ is linear in $x$,  we only need to check the boundary values of $x$.  For the left boundary ($x := \ell \hat{z}$),
\[
L(\ell \hat{z}, \hat{z}) = \hat{z} f(\ell) = \frac{\ell \hat{z} f(\ell)}{\ell} \geq f(\ell \hat{z}),
\]
where the inequality follows from Lemma \ref{lem:3sec}.  The right boundary is similar.

Now suppose that $f$ satisfies the requirements for equation (\ref{pdelta}).  Because $f$ is convex on $[\ell,u]$, $\hat{z} f(x/\hat{z})$ is convex on $[\ell \hat{z}, u \hat{z}]$.  Again we can focus on the boundary values of $x$, for which $L(x,\hat{z}) = \hat{z} f(x/\hat{z})$.

Finally, we get the volume of $\Delta_f(\ell,u)$ by looking at the absolute value of the determinant of
\[
\left(
  \begin{array}{ccc}
    \ell & \ell & u \\
   f(\ell) & f(u) & f(u) \\
    1 & 1 & 1 \\
  \end{array}
\right).
\]
\end{proof}



%
%
%

\begin{theorem} \label{thm.naive}
Suppose that $f$ is continuous, convex, and increasing on $[0,u]$, with $u> \ell > 0$
and $f(0) = 0$.
Then
\begin{align*}
\vol(\bar{S}^0_f(\ell,u)) = &\int_{0}^{f(\ell)} \left ( \int_{\frac{y}{f(u)}}^{\frac{f^{-1}(y)}{u}}\left ( uz - \ell z \right ) \;dz + \int_{\frac{f^{-1}(y)}{u}}^{\frac{f^{-1}(y)}{\ell}} \left (f^{-1}(y) - \ell z \right) \;dz \right)       \; dy\\
&+ \int_{f(\ell)}^{f(u)}\left( \int_{\frac{y}{f(u)}}^{\frac{f^{-1}(y)}{u}}\left ( uz - \ell z \right ) \;dz + \int_{\frac{f^{-1}(y)}{u}}^{1} \left (f^{-1}(y) - \ell z \right) \;dz  \;\right ) dy .
\end{align*}
\end{theorem}

\begin{proof}

We proceed using standard integration techniques, and we begin by fixing the variable $y$ and considering the corresponding 2-dimensional slice, $R_y$, of $\bar{S}^0_f(\ell,u)$.  In the $(x,z)$-space, $R_y$ is described by:
\vspace{-5pt}
\begin{multicols}{2}
\noindent
  \begin{align}
   x&\leq f^{-1}(y) \label{eq1}\\
z&\geq x/u \label{eq2} \\
z &\geq y/f(u) \label{eq3}
  \end{align}
  \begin{align}
  z &\leq x/\ell \label{eq4} \\
    z &\leq 1\label{eq5} \\
z &\geq 0\label{eq6}
  \end{align}
\end{multicols}
\vspace{-10pt}

\noindent Inequality \eqref{eq6} is implied by \eqref{eq3} because $y\geq 0$ and $f(u)>0$.  Therefore, for the various choices of $u$, $\ell$, and $y$, the tight inequalities for $R_y$ are among \eqref{eq1}, \eqref{eq2}, \eqref{eq3}, \eqref{eq4}, and \eqref{eq5}.  In fact, the region will always be described by either the entire set of inequalities  (if $y > f(\ell)$), or \eqref{eq1}, \eqref{eq2}, \eqref{eq3}, and \eqref{eq4} (if $y \leq f(\ell)$).  For an illustration of these two cases with $f(x):=x^5$, see Figures \ref{fig:21}, \ref{fig:22}.

\begin{figure}[ht!]
\centering
\caption{$f(x)=x^5$, $\ell=1$, $u=2$, $y=0.75 \leq \ell^5 = f(\ell)$}
\label{fig:21}
        \includegraphics[width=0.98\textwidth]{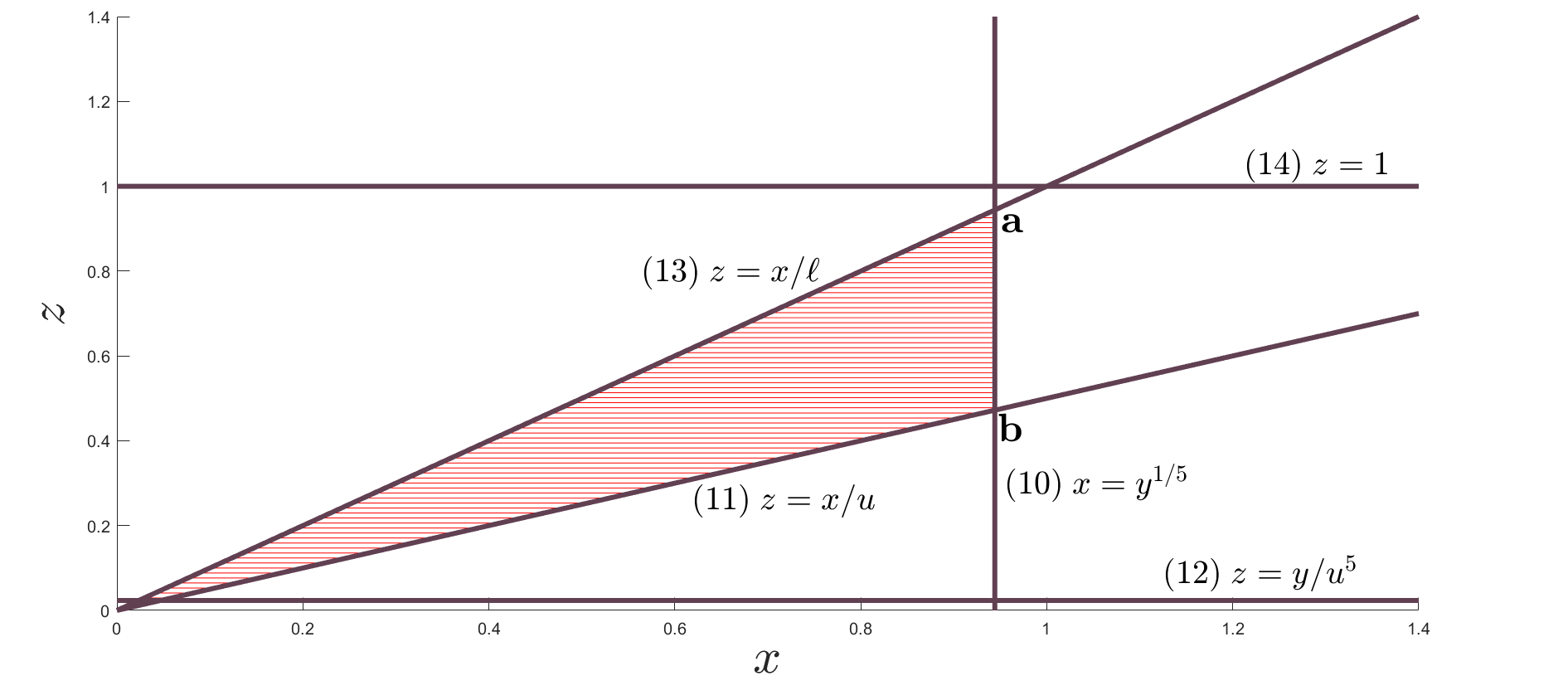}
\end{figure}

\begin{figure}[ht!]
 \centering
\caption{$f(x)=x^5$, $\ell=1$, $u=2$, $y=2 > \ell^5 = f(\ell)$}
\label{fig:22}
        \includegraphics[width=0.98\textwidth]{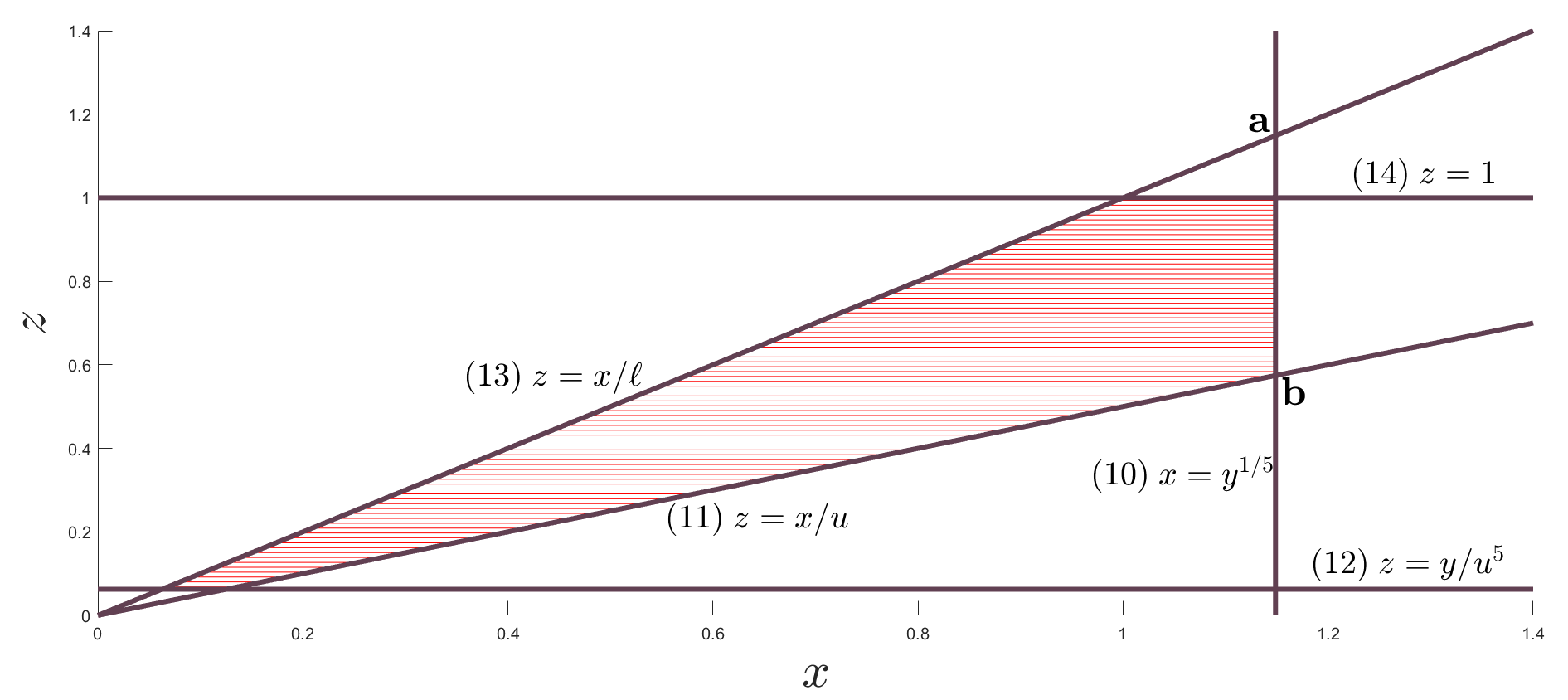}
\end{figure}

To understand why these two cases suffice, observe that together \eqref{eq2} and \eqref{eq4} create a `wedge' in the positive orthant. $R_y$ is composed of this wedge intersected with $\{(x,z) \in \mathbb{R}^2 : x \leq f^{-1}(y) \}$, for $\frac{y}{f(u)}\leq z \leq 1$. With a slight abuse of notation, based on context we use $(k)$, for $k = 10, 11, \dots, 14$, to refer both to the inequality defined above and to the 1-d \emph{boundary} of the region it describes.

Now consider the triangle formed by these `wedge' inequalities, and the inequality $x\leq f^{-1}(y)$. The vertices of this triangle are $(0,0)$, $\mathbf{a}= (x_a, z_a):=(f^{-1}(y), \frac{f^{-1}(y)}{\ell})$, and $\mathbf{b}= (x_b, z_b):=(f^{-1}(y), \frac{f^{-1}(y)}{u})$.  To understand the area that we are seeking to compute, we need to ascertain where $(0,0)$, $\mathbf{a}$, and $\mathbf{b}$ fall relative to  \eqref{eq3} and \eqref{eq5}, which bound the region $\frac{y}{f(u)}\leq z \leq 1$.  Note that the origin falls on or below \eqref{eq3}, and because $u>\ell$, $\mathbf{a}$ is always above  $\mathbf{b}$ (in the sense of higher value of  $z$).

We show that $\mathbf{b}$ must fall between the two lines \eqref{eq3} and \eqref{eq5}.  We know $y \leq f(u)$, which implies $f^{-1}(y) \leq u$ because $f$ is increasing.  Therefore $\frac{f^{-1}(y)}{u} = z_b \leq 1$.  Furthermore, because $f$ is continuous, convex, and increasing, by Lemma \ref{lem:3sec} we can see that $z_b = \frac{f^{-1}(y)}{u} \geq \frac{y}{f(u)}$.

Furthermore, given that $\mathbf{a}$ must be above $\mathbf{b}$, we now have our two cases:   $\mathbf{a}$ is either above \eqref{eq5} (if $y > f(\ell)$), or on or below \eqref{eq5} (if $y \leq f(\ell)$).

Using the observations made above, we can now calculate the area of $R_y$ via integration.  We integrate over $z$, and the limits of integration depend on the value of $y$.  If $y \leq f(\ell)$, then the area is given by the expression:
$$\int_{\frac{y}{f(u)}}^{\frac{f^{-1}(y)}{u}}\left ( uz - \ell z \right ) \;dz + \int_{\frac{f^{-1}(y)}{u}}^{\frac{f^{-1}(y)}{\ell}} \left (f^{-1}(y) - \ell z \right) \;dz$$
If $y \geq f(\ell)$, then the area is given by the expression:
$$ \int_{\frac{y}{f(u)}}^{\frac{f^{-1}(y)}{u}}\left ( uz - \ell z \right ) \;dz + \int_{\frac{f^{-1}(y)}{u}}^{1} \left (f^{-1}(y) - \ell z \right) \;dz$$
\noindent Note that when $y=f(\ell)$, these quantities are equal.

Integrating over $y$, we compute the volume of $\bar{S}^0_f(\ell,u)$ as follows:

\begin{align*}
\vol(\bar{S}^0_f(\ell,u)) = &\int_{0}^{f(\ell)} \left ( \int_{\frac{y}{f(u)}}^{\frac{f^{-1}(y)}{u}}\left ( uz - \ell z \right ) \;dz + \int_{\frac{f^{-1}(y)}{u}}^{\frac{f^{-1}(y)}{\ell}} \left (f^{-1}(y) - \ell z \right) \;dz \right)       \; dy\\
&+ \int_{f(\ell)}^{f(u)}\left( \int_{\frac{y}{f(u)}}^{\frac{f^{-1}(y)}{u}}\left ( uz - \ell z \right ) \;dz + \int_{\frac{f^{-1}(y)}{u}}^{1} \left (f^{-1}(y) - \ell z \right) \;dz  \;\right ) dy .
\end{align*}

\end{proof}

The following corollary immediately follows from Theorem \ref{thm.naive} and Lemma \ref{lem.delta}.

\begin{corollary} \label{cor:naivecap}
Suppose that $f$ is continuous, convex, and increasing on $[0,u]$, with $u> \ell > 0$
and $f(0) = 0$.
Then
\begin{align*}
\vol(\invbreve{S}^0_f(\ell,u))= &\vol(\bar{S}^0_f(\ell,u)) - \vol(\Delta_f(\ell,u)) \\
&= \int_{0}^{f(\ell)} \left ( \int_{\frac{y}{f(u)}}^{\frac{f^{-1}(y)}{u}}\left ( uz - \ell z \right ) \;dz + \int_{\frac{f^{-1}(y)}{u}}^{\frac{f^{-1}(y)}{\ell}} \left (f^{-1}(y) - \ell z \right) \;dz \right)       \; dy\\
&\quad + \int_{f(\ell)}^{f(u)}\left( \int_{\frac{y}{f(u)}}^{\frac{f^{-1}(y)}{u}}\left ( uz - \ell z \right ) \;dz + \int_{\frac{f^{-1}(y)}{u}}^{1} \left (f^{-1}(y) - \ell z \right) \;dz  \;\right ) dy \\
& \quad - \frac{1}{6}(f(u)-f(\ell))(u-\ell).
\end{align*}
\end{corollary}

In the following corollary, we consider a particular example of the exponential function first considered in Section 2.

\begin{corollary}
Let $f(x) := b^x -1$, for $b>1$, with $u> \ell > 0$. Then
\begin{align*}
\vol(\invbreve{S}^0_f(\ell,u))
&= \frac{1}{6}(u-\ell)(b^u + b^{\ell}+1)-\frac{1}{(\ln b)^2}\left(\frac{b^u-1}{u}-\frac{b^{\ell}-1}{\ell}\right). \\
\end{align*}
\end{corollary}

Putting this together with Corollary \ref{ex:expper}, we obtain:
\begin{corollary}
Let $f(x) := b^x -1$, for $b>1$, with $u> \ell > 0$. Then
\begin{multline*}
\vol(\invbreve{S}^0_f(\ell,u))- \vol(\invbreve{S}^*_{f}(\ell,u))  \\
= \frac{1}{6}(u-\ell) + \frac{1}{3 \ln b}(b^u-b^{\ell}) - \frac{1}{(\ln b)^2}\left(\frac{b^u-1}{u}-\frac{b^{\ell}-1}{\ell}\right).
\end{multline*}
 If we fix $b > 1$, and let $\ell = k u$ for some fixed $k \in (0,1)$, then
\begin{equation*}
\lim_{u \to \infty} u \times \frac{\vol(\invbreve{S}^0_{f}(\ell,u)) - \vol(\invbreve{S}^*_{f}(\ell,u)) }{\vol(\invbreve{S}^0_{f}(\ell,u)) }
=  \frac{2}{\ln b}.
\end{equation*}
\end{corollary}
\medskip

Asymptotically, with respect to  $u$, the fraction of the volume of the na\"{i}ve relaxation that is ``extra'' beyond the perspective relaxation tends to 0 rather quickly as  $u$ tends to $\infty$.
What we can see, in this asymptotic regime, is that the na\"{i}ve
relaxation is not so bad, compared to the  perspective relaxation.



\subsection{$f(\ell)>0$.}

So far, for the na\"{\i}ve relaxation, we have assumed that $f$ is defined on all of $[0,u]$
and that $f(0)=0$. Next, we extend the na\"{\i}ve relaxation and accompanying volume results to the case in which the domain of the given $f$ is considered to be restricted to $[\ell,u]$, with $u> \ell > 0$.
We wish to emphasize that the extension is relevant even when
$f$ is naturally defined at 0 and has $f(0)=0$ (e.g., $f(x)=x^p$, $p>1$).
Our strategy is to define the tightest convex under-estimator $g$ of $f$ that can be extended to pass through the origin, and then apply the na\"{i}ve relaxation to $g$. This $g$ is the convex envelope (on $[0,u]$) of the function that is $0$ at $0$ and $f(x)$ on $[\ell,u]$.
 This strategy, including our choice of $g$, is quite natural, and is relevant to modelers facing this scenario.

Recall that for real (univariate) functions, convexity implies right differentiability.  For the following lemma, we use the notation $\partial_+ f(x)$ to indicate the right derivative of $f$ at $x$.

\begin{lemma}\label{lem.defineg}
Let $f$ be positive, continuous, and convex on $[\ell,u]$ with with $u> \ell > 0$, and let $a := \max\{\hat{x} \in [\ell,u]: f(\hat{x}) \leq \hat{x} \partial_{+} f(x),  \mbox{ for all } x \in [\hat{x}, u)\}$.  Then
\[
g(x) :=
\begin{cases}
	\frac{f(a)}{a} x, & \text{if } 0 \leq x \leq a; \\
	f(x), & \text{if } a < x \leq u,
\end{cases}
\]
is the convex envelope (on $[0,u]$) of the function that is $0$ at $0$ and $f(x)$ on $[\ell,u]$.
Moreover, $g$ is  increasing (and thus invertible) on $[0,u]$.
\end{lemma}

\begin{proof}
The linear part of $g$ (defined on $[0,a]$) has the maximum possible slope to both maintain convexity (and continuity) and pass through the origin.
%
The slope of the linear part of $g$, $f(a)/a$, is positive because $0 < \ell \leq a$ and $f$ is positive on $[\ell,u]$.   This fact, along with the convexity of $g$, implies that $g$ is increasing on $[0,u]$.
\end{proof}

For a function $f$ that is continuous, positive, and convex on $[\ell,u]$, and is undefined or positive at 0, we define the na\"{i}ve relaxation of $\invbreve{D}_f(\ell,u)$ to be $\invbreve{S}^0_g(\ell,u)$, where $g$ is defined relative to $f$ as in Lemma \ref{lem.defineg}.  Volume results for $\invbreve{S}^0_g(\ell,u)$ are provided in Theorem \ref{thm.originalope}.



\begin{theorem}\label{thm.originalope}
Let $f$ be positive, continuous, and convex on $[\ell,u]$,  with $u> \ell > 0$, and define $g$ as in Lemma \ref{lem.defineg}.  Then $\invbreve{D}_f(\ell,u) \subseteq \invbreve{S}^0_g(\ell,u)$,
and
\[
\vol(\invbreve{S}^0_g(\ell,u)) = \vol(\bar{S}^0_g(\ell,u)) - \frac{1}{6}(f(u)-f(\ell))(u-\ell),
\]
where $ \vol(\bar{S}^0_g(\ell,u))$ is computed as follows.
\medskip

\noindent If $a := \ell$, then
\begin{align*}
\vol(\bar{S}^0_g(\ell,u) = &\int_{0}^{f(\ell)} \left ( \int_{\frac{y}{f(u)}}^{\frac{y\ell}{f(\ell)u}}\left ( uz - \ell z \right ) \;dz + \int_{\frac{y\ell}{f(\ell)u}}^{\frac{y}{f(\ell)}} \left (\frac{y\ell}{f(\ell)} - \ell z \right) \;dz \right)       \; dy\\
&+ \int_{f(\ell)}^{f(u)}\left( \int_{\frac{y}{f(u)}}^{\frac{g^{-1}(y)}{u}}\left ( uz - \ell z \right ) \;dz + \int_{\frac{g^{-1}(y)}{u}}^{1} \left (g^{-1}(y) - \ell z \right) \;dz  \;\right ) dy .
\end{align*}

\noindent If $a \in (\ell,u)$, then
\begin{align*}
\vol(\bar{S}^0_f(\ell,u)) = &\int_{0}^{\frac{f(a)\ell}{a}} \left ( \int_{\frac{y}{f(u)}}^{\frac{ya}{f(a)u}}\left ( uz - \ell z \right ) \;dz + \int_{\frac{ya}{f(a)u}}^{\frac{ya}{f(a)\ell}} \left (\frac{ya}{f(a)} - \ell z \right) \;dz \right)       \; dy\\
&+ \int_{\frac{f(a)\ell}{a}}^{f(a)}\left( \int_{\frac{y}{f(u)}}^{\frac{ya}{f(a)u}}\left ( uz - \ell z \right ) \;dz + \int_{\frac{ya}{f(a)u}}^{1} \left (\frac{ya}{f(a)} - \ell z \right) \;dz  \;\right ) dy \\
&+ \int_{f(a)}^{f(u)}\left( \int_{\frac{y}{f(u)}}^{\frac{g^{-1}(y)}{u}}\left ( uz - \ell z \right ) \;dz + \int_{\frac{g^{-1}(y)}{u}}^{\ell} \left (g^{-1}(y) - \ell z \right) \;dz  \;\right ) dy .
\end{align*}

\noindent If $a := u$, then
\begin{align*}
\vol(\bar{S}^0_f(\ell,u)) = &\int_{0}^{\frac{f(u)\ell}{u}} \left ( \int_{\frac{y}{f(u)}}^{\frac{yu}{f(u)\ell}} \left (\frac{yu}{f(u)} - \ell z \right) \;dz \right)       \; dy\\
&+ \int_{\frac{f(u)\ell}{u}}^{f(u)}\left( \int_{\frac{y}{f(u)}}^{\ell} \left (\frac{yu}{f(u)} - \ell z \right) \;dz  \;\right ) dy  .
\end{align*}
\end{theorem}

\begin{proof}
Apply Corollary \ref{cor:naivecap} to $g$, and then use Lemma \ref{lem.delta}.
\end{proof}

In the context of
convex-optimization solvers, instantiating the model $\invbreve{S}^0_g(\ell,u)$
\emph{generally} requires special handling for piecewise functions;
this can be addressed via coding of $g$ (function values and derivatives) for use
by NLP solvers, or through a feature of SCIP that
can accomodate piecewise-defined convex increasing functions through the modeling language AMPL
(see comments about this feature of SCIP in \cite{XuLeeSkipper2019} and in Section 6.2 of
\cite{LeeSkipper2017}).

Next, we apply Theorem \ref{thm.originalope} to $f(x) := x^p$ as a third alternative to the na\"{i}ve and the perspective relaxations.
Per Lemma \ref{lem.defineg}, we have,
\[
g(x) :=
\begin{cases}
	\ell^{p-1} x, & \text{if } 0 \leq x \leq \ell; \\
	x^p, & \text{if } \ell < x \leq u.
\end{cases}
\]
In this case, because $f(x)$ is convex on $[0,\ell]$ with $f(0) = 0$,
the linear part of $g$ provides an upper bound on $f$ over $[0,\ell]$.  Furthermore, because $g(x)$ is the point-wise maximum of $f(x)$ and
$\frac{f(\ell)}{\ell}x$, we can efficiently realize $\invbreve{S}^0_g(\ell,u) $
simply by appending $y\geq \frac{f(\ell)}{\ell}x$ to $\invbreve{S}^0_f(\ell,u)$. We have
\[
\invbreve{S}^*_f(\ell,u) \subset
\invbreve{S}^0_g(\ell,u) \subset
\invbreve{S}^0_f(\ell,u).
\]

%

We are interested in comparing the quality of relaxation the piecewise method provides compared to the perspective relaxation, so we compare the difference of each relative to the na\"{i}ve relaxation.

\begin{corollary}\label{cor:volg}
 Let $f(x) := x^p$, for $x\in [0,u]$,
with $u> \ell > 0$, and
let $g$ be defined per Lemma \ref{lem.defineg}.
Then
\begin{align*}
\frac{\invbreve{S}^0_f(\ell,u) - \invbreve{S}^0_g(\ell,u)}{\invbreve{S}^0_f(\ell,u) -\invbreve{S}^*_f(\ell,u)} 
&~=~ (p+1)\frac{1-\left(\frac{\ell}{u}\right)}{\left(\frac{u}{\ell}\right)^{p+1}-1}.
\end{align*}
\end{corollary}

\begin{proof}
Follows directly from Theorem \ref{thm.perpective_general_f}, Corollary \ref{cor:naivecap}, and Theorem \ref{thm.originalope}.
\end{proof}

Next, we will see that for power functions,
when $\ell$ is big enough relative to $u$,
much of what can be achieved by the perspective relaxation (relative to
 the na\"{i}ve relaxation) can already be achieved by the
 na\"{i}ve relaxation \emph{applied to the piecewise function $g$}
 of Lemma \ref{lem.defineg}.

\begin{corollary}\label{cor:mon}
Let $f(x) := x^p$, for $x\in [0,u]$, with $u> \ell > 0$, and
let $g$ be defined per Lemma \ref{lem.defineg}.
Then for each $p>1$ and $\phi\in(0,1)$,
 there is a $k(p,\phi)\in(0,1)$,
so that  $\ell/u$ > $k(p,\phi)$,
if and only if
\begin{align*}
\frac{\invbreve{S}^0_f(\ell,u) - \invbreve{S}^0_g(\ell,u)}{\invbreve{S}^0_f(\ell,u) -\invbreve{S}^*_f(\ell,u)}
&> \phi.
\end{align*}
\end{corollary}

\begin{proof}
Let
\[
\phi := (p+1)\frac{1-k}{1/k^{p+1}-1}.
\]
First, we note that $\lim_{k\rightarrow 1^-}~ \phi = 1$.
So, we only need to check that the continuous function $\phi$ is increasing
in $k$ on $(0,1)$.

We compute
\[
\frac{d\phi}{dk} = \frac{k^p (1 + p) (1 + k^{2 + p} + p - k (2 + p))}{(1- k^{1 + p})^2}.
\]
To see that this is positive, we only need to consider the factor
\[
j:=1 + k^{2 + p} + p - k (2 + p).
\]

For each $p$, it is clear that $j$ is continuous over $k\in(0,1)$.
The limit of $j$, as $k\rightarrow 0$ is $p+1$,
and the limit of $j$, as $k\rightarrow 1$ is $0$.
So, extending the definition of $j$ to all $k\in[0,1]$,
we have now a continuous function on $[0,1]$,
starting at value $p+1$ and ending at value $0$.
Now
\[
\frac{dj}{dk} = -2 - p + k^{1 + p} (2 + p)
= (k^{1+p}-1)(2 + p),
\]
which is negative for all $k\in(0,1)$.
Therefore $j$ is decreasing on $(0,1)$.
Because $j$ starts at a positive value ($p+1$), and decreases
all the way to 0,
it is positive on all of $(0,1)$.
Therefore, $\frac{d\phi}{dk} >0$ on $(0,1)$,
and so $\phi$ is increasing in $k$.

\end{proof}

Note that because of Corollary \ref{cor:mon} and its proof,
it is easy to compute $k(p,\phi)$ to any desired
accuracy by a univariate search method (e.g., golden-section search).
\medskip

\section{Convex power functions}
\label{sec:power}

In this section we focus on the special case of $f(x):=x^p$, for $p>1$.  For these convex
power functions, we are able to obtain not only the perspective relaxation and the na\"{i}ve relaxation, but a nested family of relaxations,
``interpolating'' between the perspective relaxation and the na\"{i}ve relaxation.

For real scalars $u> \ell > 0$ and $p>1$, we define
\begin{equation*}
S_p(\ell,u) := \left\{ (x,y,z) \in \mathbb{R}^2 \times \{0,1\} ~:~ y\geq
x^p,~ uz\geq x \geq \ell z
\right\},
\end{equation*}
and, for $0 \leq q \leq p-1$, the associated relaxations
\begin{equation*}
S^q_p(\ell,u) := \left\{ (x,y,z) \in \mathbb{R}^3 ~:~ y z^q\geq x^p,~ uz\geq x
\geq \ell z,~ 1\geq z \geq 0,~ y\geq 0
\right\}.
\end{equation*}
 For $q=p-1$, we have the perspective relaxation; that is, $S^{p-1}_p(\ell,u)
 =S^*_f(\ell,u)$, for $f(x):=x^p$.
Furthermore, if we define $0^0 =1$, then for $q=0$, we have the na\"{\i}ve relaxation;
that is, $S^0_p(\ell,u)
 =S^0_f(\ell,u)$, for $f(x):=x^p$.

Note that even though $x^p - yz^q$ is not a convex function for $q>0$ (even for $p=2,q=1$), the set
$S^q_p(\ell,u)$ \emph{is} convex. In fact, the set $S^q_p(\ell,u)$ is
higher-dimensional-power-cone  representable, which makes working with it appealing.
The \emph{higher-dimensional power cone} is defined as
\[
\mathcal{K}_{{\boldsymbol\alpha}}^{(n)} := \left\{ (x,z) \in \mathbb{R}^n_+ \times \mathbb{R} ~:~
\textstyle \prod_{j=1}^n x_j^{\alpha_j} \geq |z| \right\},
\]
where ${\boldsymbol\alpha}\in \mathbb{R}^n_+$, $\mathbf{e}'{\boldsymbol\alpha}= 1$; see \cite[Sec. 4.1.2]{Chares} and \cite[Chap. 4]{cookbook}, for example.
It is easy to check that with $x,y,z\geq 0$, we have that $yz^q \geq x^p$ is
equivalent, under the linear equation $u=1$, to
$y^{\frac{1}{p}}z^{\frac{q}{p}}u^{1-\frac{q+1}{p}} \geq x$,
which is in the format of the higher-dimensional power cone, when $q\in[1,p-1]$.

Still, computationally handling higher-dimensional power cones efficiently is not a trivial
matter, and we should not take it on without considering
alternatives.

As they are defined, both $S_p(\ell,u)$ and $S^q_p(\ell,u)$ are unbounded in the increasing $y$ direction.  However, as before, we can add a simple linear inequality to ensure that our sets are bounded.  We have
\begin{equation*}
\bar{S}^q_p(\ell,u) := \left\{ (x,y,z) \in \mathbb{R}^3 ~:~ zu^p\geq y,~ yz^q\geq x^p,~ uz\geq x
\geq \ell z,~ 1\geq z \geq 0
\right\}.
\end{equation*}

Ultimately, we are interested in the stronger upper bound on $y$ bound, yielding the set
\begin{align*}
\invbreve{S}^q_p(\ell,u) := \Big\{ (x,y,z) \in \mathbb{R}^3 ~:~ &\left(\ell^p-  \frac{u^p-\ell^p}{u-\ell} \ell\right)z + \frac{u^p-\ell^p}{u-\ell} x \geq y,
 \\& ~ yz^q\geq x^p,  ~ uz\geq x
\geq \ell z,~ 1\geq z \geq 0
\Big\}.
\end{align*}
However, as we have seen,
computing the volume is easier with the simpler bound on $y$.
We can then use Lemma \ref{lem.delta} to easily obtain the volume of $\invbreve{S}^q_p(\ell,u)$.

The following result, part of which is closely related to results  in
\cite{Akturk}, is easy to establish.
\begin{proposition}\label{lem:hull}
For $u> \ell > 0$, $p>1$, and $q\in[0,p-1]$,
(i) $\bar{S}_p(\ell,u) \subseteq \bar{S}_p^q(\ell,u)$~,
(ii) $\bar{S}_p^q(\ell,u)$ is a convex set,
(iii) $\bar{S}_p^q(\ell,u) \subseteq \bar{S}_p^{q'}(\ell,u)$, for $0\leq q' \leq q$,
and
(iv) $\conv(\bar{S_p(\ell,u)})=\bar{S}_p^{p-1}(\ell,u)$~.
\end{proposition}

\begin{proof}
(i): For $(\hat{x},\hat{y},\hat{z})\in \bar{S}_p(\ell,u)$,
it is easy to check that $(\hat{x},\hat{y},\hat{z})\in \bar{S}_p^q(\ell,u)$, considering the two cases: $\hat{z}=0$ and $\hat{z}=1$.
(ii): For $q\in [0,p-1]$, $\bar{S}_p^q(\ell,u)$ is an affine slice of a
higher-dimensional power cone.
(iii): Because $y\geq 0$, we have that $yz^q$ is non-increasing in $q$ for all $z\in[0,1]$.
(iv): By (i), we have $\bar{S}_p(\ell,u) \subseteq \bar{S}_p^{p-1}(\ell,u)$.
By (ii), we have that $\bar{S}_p^{p-1}(\ell,u)$ is a convex set.
Therefore, $\conv(\bar{S_p}(\ell,u))\subseteq \bar{S}_p^{p-1}(\ell,u)$~.
For the reverse inclusion, consider $\hat{\xi} := (\hat{x},\hat{y},\hat{z}) \in \bar{S}_p^{p-1}(\ell,u)$. If
$\hat{z} = 0$, then $\hat{\xi} = (0,0,0) \in \bar{S}_p(\ell,u) \subseteq
\conv(\bar{S}_p(\ell,u))$. If instead $0<\hat{z} \leq 1$, let $\hat{\xi}^1 :=
(\frac{\hat{x}}{\hat{z}}, \frac{\hat{y}}{\hat{z}}, 1)$. A quick
check verifies that $\hat{\xi}^1 \in \bar{S}_p(\ell,u)$ (in particular,
$\left(\frac{\hat{y}}{\hat{z}}\right) \geq \left(\frac{\hat{x}}{\hat{z}}\right)^p$ because $\hat{y}\hat{z}^{p-1} \geq \hat{x}^p$). So we have that $\hat{\xi} ~=~ (1 -
\hat{z}) (0,0,0) + \hat{z} \hat{\xi}^1~ \in~ \conv(\bar{S}_p(\ell,u))$.
\end{proof}

Note that when $q=0$, we can set $f(x) := x^p$ and use the formula obtained in Theorem~\ref{thm.naive} to obtain the volume of $\vol(\bar{S}_p^0(\ell,u))$.

\medskip
\begin{corollary}\label{cor:nai_pow} For $p > 1$, and $u> \ell > 0$,
\[\vol(\bar{S}_p^0(\ell,u)) =  \frac{\left(p^2+3 p-1\right) u^{p+1}+3 {\ell}^{p+1}- \left(p^2+3 p+2\right){\ell} u^p}{3 (p+1) (p+2)}.\]

\end{corollary}
\medskip

Extending Corollary \ref{cor:nai_pow}, we have the following result.

\begin{theorem}\label{thm:vol}
For $p>1$, $0\leq q \leq p-1$, and $u> \ell > 0$,
 \[\vol(\bar{S}_p^q(\ell,u))= \frac{(p^2-pq+3p-q-1)u^{p+1}+3\ell^{p+1}-(p+1)(p-q+2)\ell u^p}{3(p+1)(p-q+2)}.
 \]
\end{theorem}

\begin{proof}
The case of $q=0$ is Corollary \ref{cor:nai_pow}, so we can assume that
$q>0$.
The proof structure of Theorem~\ref{thm:vol} is similar to that of Theorem~\ref{thm.naive}, therefore, once again we proceed using standard integration techniques.  We fix the variable $y$ and consider the corresponding 2-dimensional slice, $R_y$, of $\bar{S}_p^q(\ell,u)$.  In the $(x,z)$-space, $R_y$ is described by:
\vspace{-5pt}
\begin{multicols}{2}
\noindent
  \begin{align}
   z&\geq x^{p/q}y^{-1/q} \label{eq1*}\\
z&\geq x/u \label{eq2*} \\
z &\geq y/u^p \label{eq3*}
  \end{align}
  \begin{align}
  z &\leq x/\ell \label{eq4*} \\
    z &\leq 1\label{eq5*} \\
z &\geq 0\label{eq6*}
  \end{align}
\end{multicols}
\vspace{-10pt}

\noindent These inequalities are those in Theorem~\ref{thm.naive} with the exception of \eqref{eq1*}, and, analogous to the proof of Theorem~\ref{thm.naive}, the tight inequalities for $R_y$ are among \eqref{eq1*}, \eqref{eq2*}, \eqref{eq3*}, \eqref{eq4*}, and \eqref{eq5*} for all choices of $p$, $q$, $\ell$, $u$, and $y$.  Furthermore, the region will always be described by either the entire set of inequalities  (if $y > \ell^p$), or \eqref{eq1*}, \eqref{eq2*}, \eqref{eq3*}, and \eqref{eq4*} (if $y \leq \ell^p$).  For an illustration of these two cases with $p=5$ and $q=3$, see Figures \ref{fig:11}, \ref{fig:12}.

\begin{figure}[ht!]
\centering
\caption{$p=5$, $q=3$, $\ell=1$, $u=2$, $y=0.75 \leq \ell^p $}
\label{fig:11}
        \includegraphics[width=0.98\textwidth]{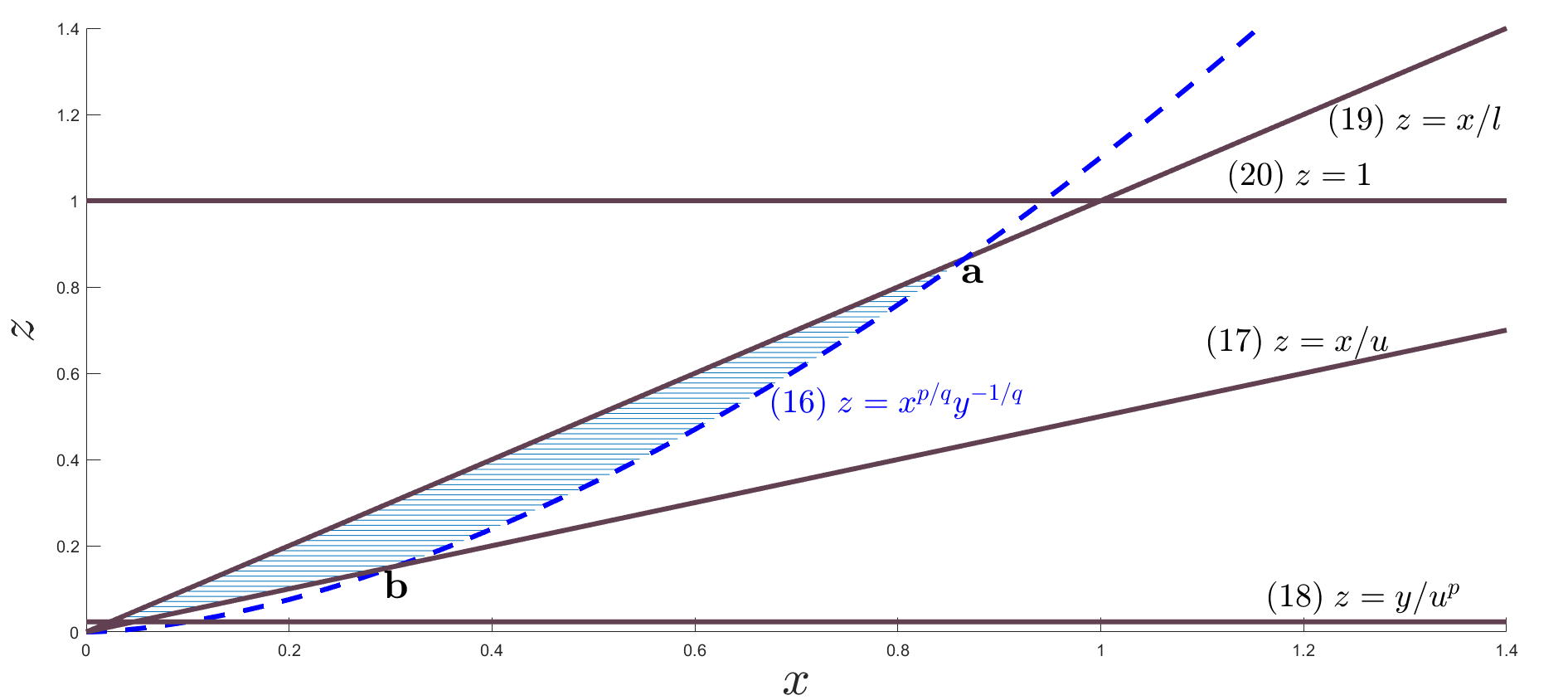}
\end{figure}

\begin{figure}[ht!]
 \centering
\caption{$p=5$, $q=3$, $\ell=1$, $u=2$, $y=2 > \ell^p $}
\label{fig:12}
        \includegraphics[width=0.98\textwidth]{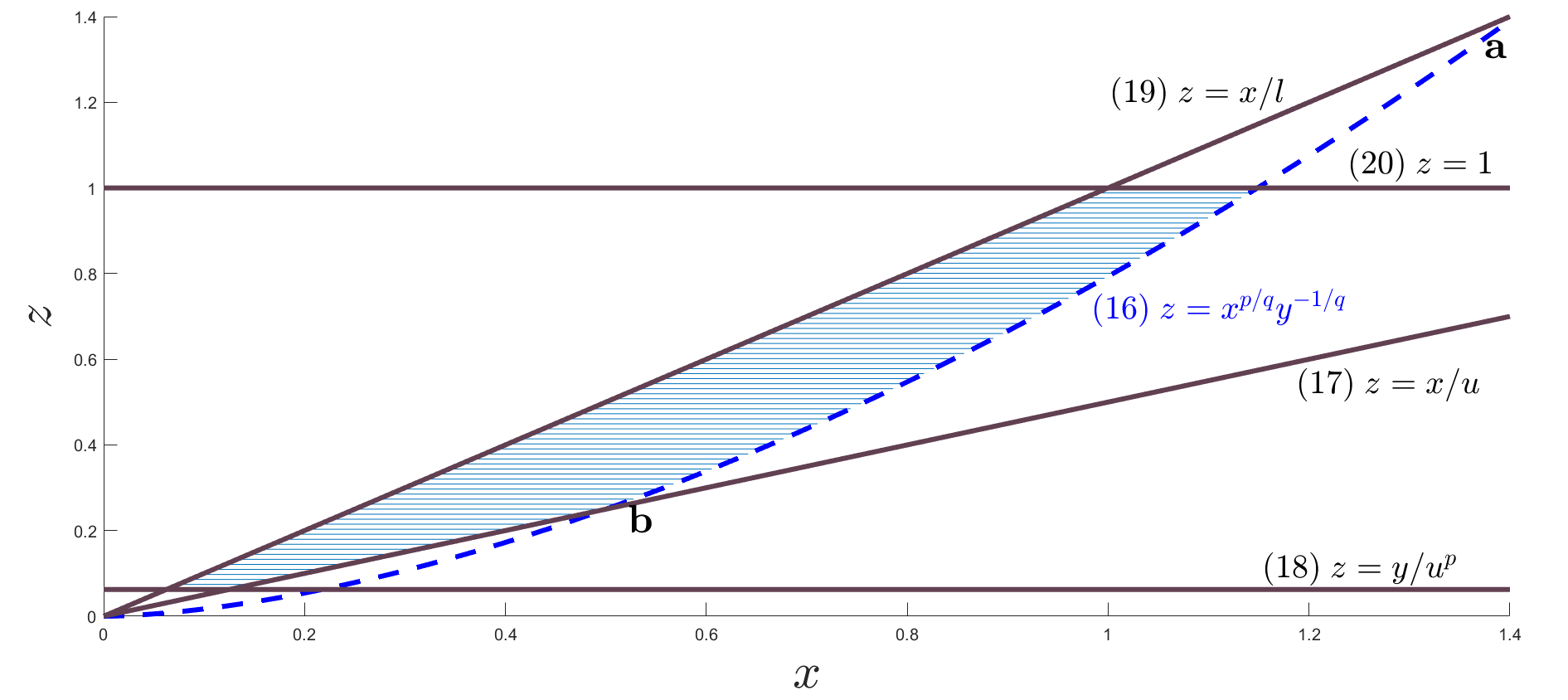}
\end{figure}

To understand why these two cases suffice, recall that together \eqref{eq2*} and \eqref{eq4*} create a `wedge' in the positive orthant.  $R_y$ is composed of this wedge intersected with $\{(x,z) \in \mathbb{R}^2 : z\geq x^{p/q}y^{-1/q} \}$, for $\frac{y}{u^p}\leq z \leq 1$.  As before, we slightly abuse notation and based on context we use $(k)$, for $k = 16, 17, \dots, 20$, to refer both to the inequality defined above and to the 1-d \emph{boundary} of the region it describes.

We consider the set of points formed by the wedge and the inequality $z\geq x^{p/q}y^{-1/q}$.  Curves \eqref{eq1*} and \eqref{eq4*} intersect at $(0,0)$ and $\mathbf{a} = (x_a, z_a) :=\left(\left(\frac{y}{\ell^q}\right)^{1/(p-q)},\right.$ $\left.\left(\frac{y}{\ell^p}\right)^{1/(p-q)}\right)$.   Curves \eqref{eq1*} and \eqref{eq2*} intersect at $(0,0)$ and $\mathbf{b} = (x_b, z_b) :=\left(\left(\frac{y}{u^q}\right)^{1/(p-q)},\right.$ $\left.\left(\frac{y}{u^p}\right)^{1/(p-q)}\right)$.   
As before, to understand the area that we are seeking to compute, we need to ascertain where $(0,0)$, $\mathbf{a}$, and $\mathbf{b}$ fall relative to  \eqref{eq3*} and \eqref{eq5*}, which bound the region $\frac{y}{u^p}\leq z \leq 1$.  Note that the origin falls on or below \eqref{eq3*}, and because $u>\ell$, $\mathbf{a}$ is always above  $\mathbf{b}$ (in the sense of higher value of  $z$).

We show that $\mathbf{b}$ must fall between lines \eqref{eq3*} and \eqref{eq5*}.  This is equivalent to $\frac{y}{u^p} \leq \left(\frac{y}{u^p}\right)^{1/(p-q)}= z_b \leq 1$.  Now, we know $y \leq u^p$, which implies $\frac{y}{u^p} \leq 1$.  From our assumptions on $p$ and $q$, we also have $0 < \frac{1}{p-q} \leq 1$.  From this we can immediately conclude $\frac{y}{u^p} \leq \left(\frac{y}{u^p}\right)^{1/(p-q)} = z_b \leq 1$.
%

Furthermore, given that $\mathbf{a}$ must be above $\mathbf{b}$, we now have our two cases:   $\mathbf{a}$ is either above \eqref{eq5*} (if $y > \ell^p$), or on or below \eqref{eq5*} (if $y \leq \ell^p$).   
Using the observations made above, we can now calculate the area of $R_y$ via integration.  We integrate over $z$, and the limits of integration depend on the value of $y$.  If $y \leq \ell^p$, then the area is given by the expression:
$$\int_{\frac{y}{u^p}}^{z_b}\left ( uz - \ell z \right ) \;dz + \int_{z_b}^{z_a} \left ((yz^q)^{\frac{1}{p}} - \ell z \right) \;dz.$$
If $y \geq \ell^p$, then the area is given by the expression:
$$\int_{\frac{y}{u^p}}^{z_b}\left ( uz - \ell z \right ) \;dz + \int_{z_b}^{1} \left ((yz^q)^{\frac{1}{p}} - \ell z \right) \;dz.$$
\noindent Note that when $y=\ell^p$, these quantities are equal.  Furthermore, when $q=p-1$ (and we have the hull), the first integral in each sum is equal to zero.

Integrating over $y$, we compute the volume of $\bar{S}_p^q(\ell,u)$ as follows:
\begin{align*}
&\int_{0}^{\ell^p} \left (\int_{\frac{y}{u^p}}^{\left(\frac{y}{u^p}\right)^{1/(p-q)}}\left ( uz - \ell z \right ) \;dz + \int_{\left(\frac{y}{u^p}\right)^{1/(p-q)}}^{\left(\frac{y}{l^p}\right)^{1/(p-q)}} \left ((yz^q)^{\frac{1}{p}} - \ell z \right) \;dz  \right)       \; dy\\
&+ \int_{\ell^p}^{u^p}\left( \int_{\frac{y}{u^p}}^{\left(\frac{y}{u^p}\right)^{1/(p-q)}}\left ( uz - \ell z \right ) \;dz + \int_{\left(\frac{y}{u^p}\right)^{1/(p-q)}}^{1} \left ((yz^q)^{\frac{1}{p}} - \ell z \right) \;dz  \;\right ) dy \\
&= \frac{(p^2-pq+3p-q-1)u^{p+1}+3\ell^{p+1}-(p+1)(p-q+2)\ell u^p}{3(p+1)(p-q+2)}.
\end{align*}
\end{proof}

By looking at these slices in the $(x,z)$ plane, we are able to gain a better geometric understanding of how the relaxation is tightening as $q$ increases from $0$ to $p-1$.  See Figure \ref{fig:varyingq} to see how the relaxation improves as $q$ increases. The figure is drawn for the case where $y\geq \ell^p$ but the picture would be very similar for the alternative case.  When $q=p-1$ (and we have precisely the convex hull), inequality \eqref{eq2} becomes redundant; this is equivalent to when we noted in the proof that the first integral is equal to zero when we have the hull.

\begin{figure}[ht]
\centering
 \caption{Increasing $q$ tightens the relaxation.}
\label{fig:varyingq}
        \includegraphics[width=0.99\textwidth]{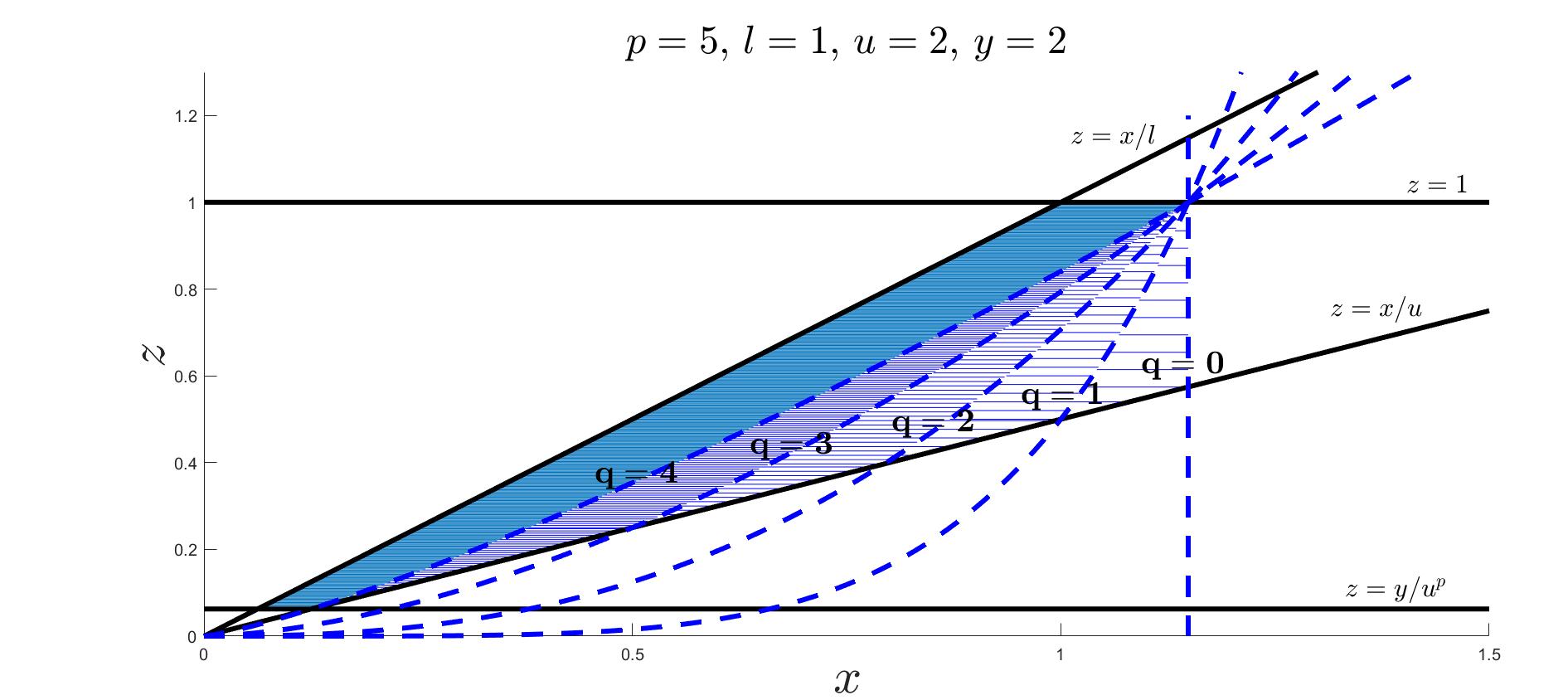}
\end{figure}

Now the following corollary  follows from Theorem~\ref{thm:vol} (and its proof)
 and (the proof of)  Lemma~\ref{lem.delta}.

\begin{corollary}\label{cor:volcapSpq}
For $p>1$, $0\leq q \leq p-1$, , and $u> \ell > 0$,
 \begin{align*}
&\vol(\invbreve{S}_p^q(\ell,u)) =  \\
& \frac{(p^2-pq+3p-q-1)u^{p+1}+3\ell^{p+1}-(p+1)(p-q+2)\ell u^p}{3(p+1)(p-q+2)}
-\frac{(u^p-\ell^p)(u-\ell)}{6}.
 \end{align*}
\end{corollary}

Note that for $q=p-1$, this is exactly what we get by plugging $f(x)=x^p$ into Theorem~\ref{thm.perpective_general_f}.

We can use the volume formula of Corollary \ref{cor:volcapSpq}
to compare relaxations.
 For the most general case, we have the following.

\begin{corollary}\label{cor:volq1q2}
For $p > 1$, $0 \leq q_1 < q_2 \leq p-1$, and $u> \ell > 0$,
\[
\vol(\invbreve{S}_p^{q_1}(\ell,u)) - \vol(\invbreve{S}_p^{q_2}(\ell,u)) =
\frac{(q_2-q_1)(u^{p+1}-\ell^{p+1})}{(p+1)(p-q_1+2)(p-q_2+2)}.
\]
\end{corollary}

\noindent Applying this result, we can precisely quantify how much better the  perspective relaxation
($q=p-1$) is compared to the  na\"{\i}ve relaxation ($q=0$):
\begin{corollary}\label{cor:naivediff}
For $p>1$, and $u> \ell > 0$,
\[
\vol(\invbreve{S}_p^0(\ell,u)) - \vol(\invbreve{S}_p^{p-1}(\ell,u)) = \frac{(p-1)(u^{p+1}-\ell^{p+1})}{3(p+1)(p+2)}
\quad \left[=\frac{u^{3}-\ell^{3}}{36}, \mbox{ for } p=2 \right].
\]
\end{corollary}

\noindent It is also interesting to quantify how much better the  perspective relaxation ($q=p-1$) is compared to the ``na\"{\i}ve perspective relaxation''
($q=1$):
\begin{corollary}
For $p> 2$, and $u> \ell > 0$,
\begin{equation*}\vol(\invbreve{S}_p^1(\ell,u)) - \vol(\invbreve{S}_p^{p-1}(\ell,u)) =  \frac{(p-2)(u^{p+1}-\ell^{p+1})}{3(p+1)^2}.\end{equation*}
\end{corollary}


To see if a tighter relaxation is worth the extra computational effort, it is useful to see what proportion of the volume is removed when replacing a weaker relaxation to a stronger one.  Our volume formula leads to a lower bound, in terms of $p$, $q_1$ and $q_2$, on the proportion of the excess volume removed when replacing $\invbreve{S}_p^{q_1}(\ell,u)$ by $\invbreve{S}_p^{q_2}(\ell,u)$, for $q_1 < q_2$.  Interestingly:
\begin{itemize}
\item The exact formula only depends on $\ell$ and $u$ through their ratio $k$.
\item We have a lower bound that is completely independent of $\ell$ and $u$.
\end{itemize}

\begin{corollary}\label{cor:ratioq1q2}
For $p > 1$, $0 \leq q_1 < q_2 \leq p-1$, $u> \ell > 0$, and $k:=\ell/u$,
%
%
\begin{align*}
&\frac{\vol(\invbreve{S}_p^{q_1}(\ell,u)) - \vol(\invbreve{S}_p^{q_2}(\ell,u))}{\vol(\invbreve{S}_p^{q_1}(\ell,u))} \\
&\quad= \frac{6(q_2-q_1)\left(\frac{1-k^{p+1}}{(1-k)(1+k^p)}\right)}
	{(p-q_2+2)(p^2+3p-q_1(p+1)-1)-3(p-q_2+2)\left(\frac{(1+k)(1-k^p)}{(1-k)(1+k^p)}\right)} \\
&\quad\geq  \frac{6(q_2-q_1)}
	{(p-q_2+2)(p^2+3p-q_1(p+1)-4)}.
\end{align*}
Moreover, the inequality becomes tight only as $k \rightarrow 0$.
\end{corollary}

\begin{proof}
Replacing $\ell$ with $ku$ and combining terms in the volume expression in Corollary \ref{cor:volcapSpq}, we see that $u$ arises only in a factor of $u^{p+1}$ of the entire volume expression:
\begin{multline*}
\vol(\invbreve{S}_p^{q_1}(\ell,u)) = \\ \frac{u^{p+1}\left[2(p^2 +3p - pq_1 - q_1 - 1) + 6k^{p+1} - (p+1)(p-q_1+2)(k^{p+1} - k^p +k +1)\right]}{6(p+1)(p-q_1+2)}.
\end{multline*}
Similarly, replacing $\ell$ with $ku$ in the difference of volumes expression in Corollary \ref{cor:volq1q2}, we obtain the same $u^{p+1}$ factor,
\begin{equation*}
\vol(\invbreve{S}_p^{q_1}(\ell,u)) - \vol(\invbreve{S}_p^{q_2}(\ell,u)) =
\frac{u^{p+1}(q_2-q_1)(1-k^{p+1})}{(p+1)(p-q_1 + 2)(p-q_2+2)},
\end{equation*}
so that in the ratio, the $u$ factors cancel:
\begin{align*}\displaystyle
	&\frac{\vol(\invbreve{S}_p^{q_1}(\ell,u)) - \vol(\invbreve{S}_p^{q_2}(\ell,u))}{\vol(\invbreve{S}_p^{q_1}(\ell,u))} =\\
& \frac{6(q_2-q_1)(1-k^{p+1})}
				{(p-q_2+2)\left[  2(p^2-pq_1+3p-q_1-1)+6k^{p+1}-(p+1)(p-q_1+2)(k^{p+1} - k^p +k +1)  \right]}.
\end{align*}

Letting $A:=k^{p+1} - k^p +k +1$ and simplifying further, we obtain:

\begin{align*}\displaystyle
	&\frac{\vol(\invbreve{S}_p^{q_1}(\ell,u)) - \vol(\invbreve{S}_p^{q_2}(\ell,u))}{\vol(\invbreve{S}_p^{q_1}(\ell,u))} \\
        =& \frac{6(q_2-q_1)(1-k^{p+1})}
				{(p-q_2+2)\left[  2(p^2-pq_1+3p-q_1-1)
				 +(6k^{p+1}-Ap^2+Apq_1-3Ap+Aq_1-2A)\right]} \\
		=& \frac{6(q_2-q_1)(1-k^{p+1})}
				{(p-q_2+2)\left[  2(p^2-pq_1+3p-q_1-1)
				 -A (p^2-pq_1+3p-q_1-1) - (3A-6k^{p+1})\right]} \\
		=& \frac{6(q_2-q_1)(1-k^{p+1})}
				{(p-q_2+2)\left[  (2-A)(p^2-pq_1+3p-q_1-1)
				  -3(A-2k^{p+1})\right]} \\
             =& \frac{6(q_2-q_1)\left(\frac{1-k^{p+1}}{2-A}\right)}
				{(p-q_2+2)(p^2-pq_1+3p-q_1-1) - 3(p-q_2+2)\left(\frac{A - 2k^{p+1}}{2-A}\right)},
\end{align*}
which is the second expression in the statement of the corollary when $A$ is replaced by its equivalent expression.

Recalling that $0 < k < 1$ and $p > 1$, it is clear that
\[
2 - A = (1-k)(1+k^p) > 0
\]
and
\[
2 - A = (1-k^p)-k(1-k^{p-1}) \leq 1-k^p,
\]
so that
\[
\frac{1-k^{p+1}}{2-A} \geq 1.
\]
Similarly,
\[
\frac{A - 2k^{p+1}}{2-A} = \frac{(1-k^{p+1}) + k(1-k^{p})}{(1-k^{p+1})-k(1-k^{p})} \geq 1,
\]
and the lower bound follows.
When the ratio of volumes is expressed in this form, it is clear that the bound is tight if $k = 0$ and strict if $k>0$.
\end{proof}


Applying this result with $q_1 = 0$ and $q_2 = p-1$ demonstrates that
the excess volume of the na\"{\i}ve relaxation, as compared to the perspective relaxation, is substantial.


\begin{corollary}  For $p > 1$,  $u> \ell > 0$, and $k:=\ell/u$,
\[
\frac{\vol(\invbreve{S}_p^0(\ell,u)) - \vol(\invbreve{S}_p^{p-1}(\ell,u))}{\vol(\invbreve{S}_p^0(\ell,u))}
  =\frac{2(p-1)\left(\frac{1-k^{p+1}}{(1-k)(1+k^p)}\right)}{(p^2+3p-1) - 3\left(\frac{(1+k)(1-k^p)}{(1-k)(1+k^p)}\right)} \geq \frac{2}{p+4}~.
\]
Moreover, the inequality becomes tight only as $k \rightarrow 0$.
\end{corollary}

\section{Computational experiments}\label{sec:comp}


Some earlier theoretical results on using volume as a measure to
compare relaxations were substantiated by computational experiments.
Results in \cite{LM1994} concerning facility location were
computationally substantiated in \cite{Lee2007}. Results in
\cite{SpeakmanLee2015} concerning models for triple products
were experimentally validated in \cite{SpeakmanYuLee}.

In this section we describe some illustrative computations
related to our present work.
This is not meant to be a thorough  computational study. Rather,
we simply wish to illustrate what our theoretical results
can predict about actual computational tradeoffs.

We carried out experiments on a 16-core machine (running Windows Server 2012 R2):
two Intel Xeon CPU E5-2667 v4 processors
running at 3.20GHz, with 8 cores each, and 128 GB of memory.
We used the conic solver SDPT3 4.0 (\cite{sdpt3}) under the Matlab ``disciplined convex optimization'' package CVX (\cite{cvx}).

\subsection{Separable quadratic-cost knapsack covering.}

Our first experiment is based on the following model, which we think of as a relaxation of the
identical model having the constraints $z_i\in\{0,1\}$ for $i=1,2,\ldots,n$.
The data $\mathbf{c},\mathbf{f},\mathbf{a},\mathbf{l},\mathbf{u}\in \mathbb{R}^n$ and $b\in\mathbb{R}$  are all positive.
The idea is that we have  costs $c_i$ on $x_i^2$, and $x_i$ is either 0 or in
the ``operating range'' $[\ell_i,u_i]$. We pay a cost $f_i$ when $x_i$ is nonzero.

\begin{align*}
&\min\ \mathbf{c}'\mathbf{y} + \mathbf{f}'\mathbf{z}\\
&\quad \mbox{subject to:}\\
&\mathbf{a}'\mathbf{x} \geq b~;\\
&u_i z_i \geq x_i \geq \ell_i z_i~,\qquad i=1,\ldots,n~;\\
&u_i^2 z \geq y_i \geq  x_i^2~, \qquad i=1,\ldots,n~;\\
& 1\geq z_i \geq 0,  \qquad i=1,\ldots,n~.
\end{align*}

In our experiment, we independently and randomly chose
 $a_i\sim\mathcal{U}(1.0,1.2)$,
 $f_i\sim\mathcal{U}(10.0,10.2)$,
 $\ell_i\sim\mathcal{U}(0,20)$,
 $\delta_i\sim\mathcal{U}(10,11)$,
 $u_i:=\ell_i+\delta_i$,
 $c_i\sim\mathcal{U}(0,1)$,
 $b:=\mathbf{a}'(\boldsymbol{\ell}+\frac{1}{4}\boldsymbol{\delta})$.
 We purposely chose most of the parameters and
 the ranges $u_i-\ell_i=\delta_i$ to have very low variance,
and we took $n=30,000$ so that we would get a strong averaging behavior.
In this way, we sought to focus on the
dependence of our results on the values of the  $\ell_i$ and $u_i$,
but not on the difference $u_i-\ell_i$. Corollary \ref{cor:naivediff}
predicts a monotone dependence on $u_i^3-\ell_i^3$, and this is
what we sought to illustrate.

For \emph{some} of the $i$, we conceptually replace $y_i \geq  x_i^2$ with
its perspective tightening $y_i z_i \geq x_i^2$, $y\geq 0$;
really, we are using a conic solver,
so we instead employ an SOCP representation. We do this for the choices of
$i$ that are the $k$ highest according to a ranking of all $i$,
$1\leq i \leq n$. We let $k=n(j/15)$, with $j=0,1,2,\ldots,15$.
Denoting the polytope with no tightening by $Q$ and with tightening by $P$,
we looked at three different rankings: descending values of
$\vol(Q)-\vol(P)=(u_i^3-\ell_i^3)/36$,  ascending values of
$\vol(Q)-\vol(P)$, and random. For $n=30,000$, we present our results in
Figure \ref{fig:biglips1}.
As a baseline, we can see that if we only want to
apply the perspective relaxation for some pre-specified fraction of the
$i$'s, we get the best improvement in the objective value (thinking of it as a lower bound for the
true problem with the constraints $z_i\in\{0,1\}$) by preferring $i$ with the largest value
of $u_i^3-\ell_i^3$. Moreover, most of the benefit is already achieved at much lower
values of $k$ than for the other rankings.

\begin{figure}[!htb]
\caption{Separable-quadratic knapsack covering, $n=30,000$}
\label{fig:biglips1}
 \begin{center}
 \includegraphics[width=\textwidth]{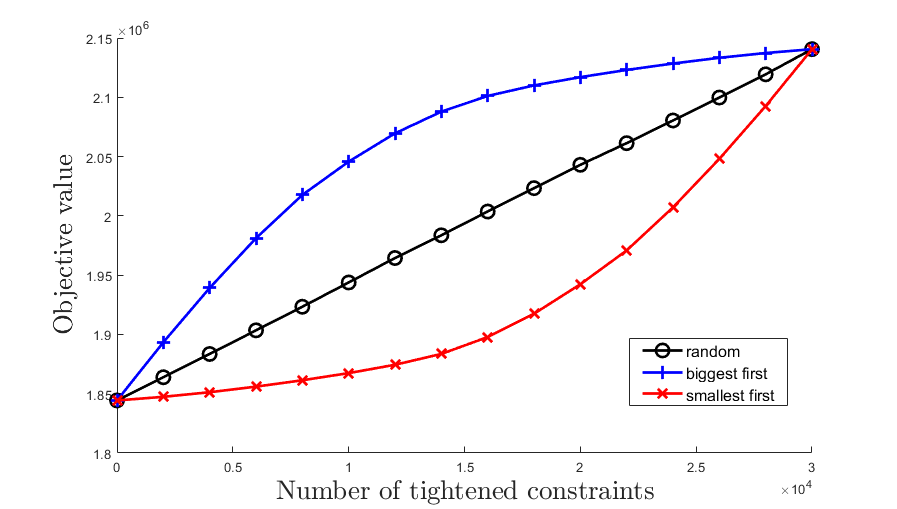}
 \end{center}
\end{figure}

\cite{LM1994} and \cite{SpeakmanYuLee} suggested that for a pair of relaxations $P, Q\subset \mathbb{R}^d$, a
good measure for evaluating $Q$ relative to $P$ might be
$\sqrt[\leftroot{-3}\uproot{3}d]{\vol(Q)}-\sqrt[\leftroot{-3}\uproot{3}d]{\vol(P)}$ (in our present setting, we have $d=3$).
We did experiments ranking by this, rather then the simpler $\vol(Q)-\vol(P)$, and we
found no significant difference in our results. This can be explained by the fact that ranking by either of these choices is very similar for our test set.
In Figure \ref{fig:scatter},
 we have a scatter  plot of $\vol(Q)-\vol(P)$ vs.
$\sqrt[\leftroot{-3}\uproot{3}d]{\vol(Q)}-\sqrt[\leftroot{-3}\uproot{3}d]{\vol(P)}$,
across the $n=30,000$ choice of $u_i$ and $\ell_i$ from our experiment above.
We can readily see that ranking by either of these choices is very similar.

\begin{figure}[!htb]
\caption{How to use volumes: Separable-quadratic knapsack covering, $n=30,000$}
\label{fig:scatter}
 \begin{center}
 \includegraphics[width=\textwidth]{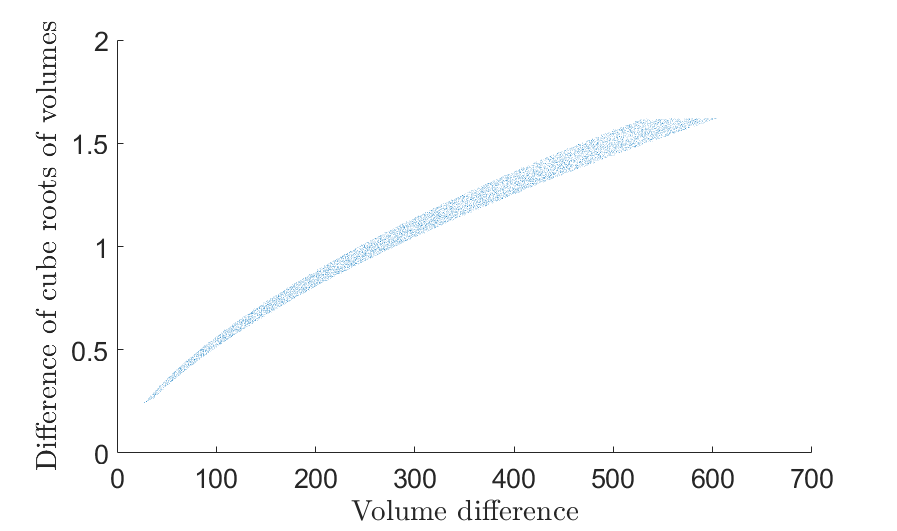}
 \end{center}
\end{figure}

\noindent More precisely, the Kendall and Spearman rank correlation coefficients (``Kendall's $\tau$'' and ``Spearman's $\rho$'':  non-parametric
statistics used to measure  ordinal association between two measured quantities) are  0.9647
and 0.9984, respectively. Notably the same numbers, measuring a relationship between
either of our two quantities and the $u_i-\ell_i$ are all below 0.07, indicating that
$u_i-\ell_i$ is not a good substitute for our measures based on volumes.
Of course this is
related to the fact that we (purposely) generated our data so that the $u_i-\ell_i$ has low variation.

\subsection{Mean-variance optimization.}

Next, we conducted a similar experiment on a richer model,
though at a smaller scale.
Our model is for a (Markowitz-style) mean-variance optimization problem
(see \cite{gunlind1} and \cite{Frangioni2006}).
We have $n$ investment vehicles.
The vector $\mathbf{a}$ contains the expected returns for the portfolio/holdings
$\mathbf{x}$. The scalar $b$ is our requirement
for the minimum expected return of our portfolio.
Asset $i$ has a possible range $[\ell_i,u_i]$, and we limit the number of assets that
we hold to $\kappa$.

Variance is measured, as usual, via a quadratic
which is commonly taken to have the form: $\mathbf{x}'\left(\mathbf{Q}+\Diag(\mathbf{c})\right)\mathbf{x}$,
where $\mathbf{Q}$ is positive definite and $\mathbf{c}$ is all positive (see \cite{gunlind1} and \cite{Frangioni2006} for details
on why this form is used in the application). Taking the Cholesky factorization
$\mathbf{Q}=\mathbf{M}\mathbf{M}'$, we define $\mathbf{w}:=\mathbf{M}'\mathbf{x}$,
and introduce the scalar variable $v$. In this way, we arrive at the model:

\begin{align*}
&\min\ v + \mathbf{c}'\mathbf{y} \\
&\quad\mbox{subject to:}\\
&\mathbf{a}'\mathbf{x} \geq b~;\\
&\mathbf{e}'\mathbf{z} \leq \kappa~;\\
&\mathbf{w}-\mathbf{M}'\mathbf{x} = \mathbf{0}~;\\
& v \geq \|\mathbf{w}\|^2~;\\
& u_i z_i \geq x_i \geq \ell_i z_i~,\qquad i=1,\ldots,n~;\\
& u_i^2 z \geq y_i \geq  x_i^2~, \qquad i=1,\ldots,n~;\\
&
1\geq z_i \geq 0,  \qquad i=1,\ldots,n~;\\
&w_i \mbox{ unrestricted}, \qquad i=1,\ldots,n~.
\end{align*}

\noindent Note: The inequality $v \geq \|\mathbf{w}\|^2$ is correct; there is a typo in \cite{gunlind1},
where it is written as $v \geq \|\mathbf{w}\|$. The inequality $v \geq \|\mathbf{w}\|^2$,
while not formulating a Lorentz (second-order) cone,
may be re-formulated
as an affine slice of a rotated Lorentz cone, or not, depending
on the solver employed.

Most of our parameters are the same as for our first experiment.
Here we took $n=1,500$ (these are harder models), and $\kappa:=\lfloor0.8n \rfloor$.
The remaining parameter is the lower-triangular matrix $M$, which we took
to have independent entries distributed as $\mathcal{U}(0,0.0025)$.

Our results,
summarized in Figure \ref{fig:biglips2},
follow the same general trend as Figure \ref{fig:biglips1},
again agreeing with the prediction of Corollary \ref{cor:naivediff}.

\begin{figure}[!htb]
\caption{Mean-variance optimization: $n=1,500$}
\label{fig:biglips2}
 \begin{center}
\includegraphics[width=\textwidth]{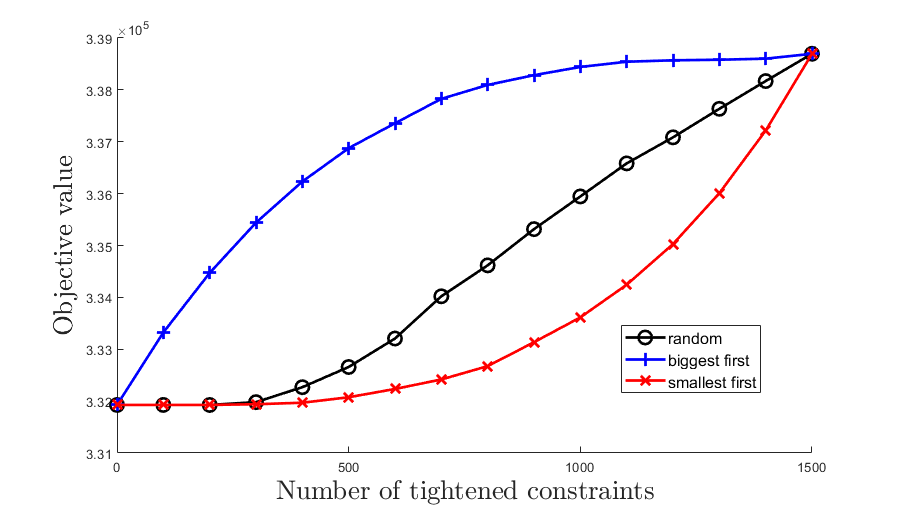}
 \end{center}
\end{figure}

 \FloatBarrier


\section{Concluding remarks}\label{sec:conc}

In general, we have presented three possibilities of relaxation:
\begin{itemize}
\item The weakest is the na\"{i}ve relaxation, and it is the least
computationally burdensome, amenable to general NLP solvers (like Ipopt or Knitro).
\item Next tightest is the na\"{i}ve relaxation applied to the tightened function
$g$ (as defined in Lemma \ref{lem.defineg}). This one, in the context of
convex-optimization solvers, sometimes requires special handling for piecewise functions
(see comments between Theorem \ref{thm.originalope} and Corollary \ref{cor:volg}).
\item Finally, the tightest is the perspective relaxation, but this feature
best lends itself to reliable solution using conic solvers (like Mosek and SDPT3),
and only when the appropriate cone is supported by the solver.
Presently, restriction to conic solvers limits the kinds of other functions
that can be present in a model.
\end{itemize}
Furthermore, for power functions, we give a continuum of further possibilities, also
best exploited using conic solvers.

As the solver landscape is rapidly changing,
we do not have a uniform answer to the question of which model and solver to
use. As we have demonstrated with our limited experiments, even for a fixed solver,
our results can give some useful guidance in model choice.
Going forward, we do hope and believe that our approach can be used by modelers and solver developers in their work.

%

\bibliographystyle{spmpsci}      

\bibliography{perspecbib}

\begin{thebibliography}{10}
\providecommand{\url}[1]{{#1}}
\providecommand{\urlprefix}{URL }
\expandafter\ifx\csname urlstyle\endcsname\relax
  \providecommand{\doi}[1]{DOI~\discretionary{}{}{}#1}\else
  \providecommand{\doi}{DOI~\discretionary{}{}{}\begingroup
  \urlstyle{rm}\Url}\fi

\bibitem{Akturk}
Akt\"{u}rk, M.S., Atamt\"{u}rk, A., G\"{u}rel, S.: A strong conic quadratic
  reformulation for machine-job assignment with controllable processing times.
\newblock Operations Research Letters \textbf{37}(3), 187--191 (2009)

\bibitem{BCSZ}
Basu, A., Conforti, M., Di~Summa, M., Zambelli, G.: Optimal cutting planes from
  the group relaxations.
\newblock Mathematics of Operations Research \textbf{44}(4), 1208--1220 (2019)

\bibitem{Bertsekas}
Bertsekas, D.P.: Nonlinear programming, third edn.
\newblock Athena Scientific Optimization and Computation Series. Athena
  Scientific, Belmont, MA (2016)

\bibitem{Chares}
Chares, P.R.: Cones and interior-point algorithms for structured convex
  optimization involving powers and exponentials.
\newblock Docteur en {S}ciences de l{’}{I}ng\'enieur, Universit\'e catholique
  de Louvain (2007)

\bibitem{PClark}
Clark, P.L.: Honors Calculus (2014).
\newblock \url{math.uga.edu/~pete/2400full.pdf}

\bibitem{DM}
Dey, S.S., Molinaro, M.: Theoretical challenges towards cutting-plane
  selection.
\newblock Mathematical Programming \textbf{170}(1), 237--266 (2018)

\bibitem{Frangioni2006}
Frangioni, A., Gentile, C.: Perspective cuts for a class of convex 0--1 mixed
  integer programs.
\newblock Mathematical Programming \textbf{106}(2), 225--236 (2006)

\bibitem{cvx}
Grant, M., Boyd, S.: {CVX}: Matlab software for disciplined convex programming,
  version 2.1, build 1123.
\newblock \url{http://cvxr.com/cvx} (2017)

\bibitem{gunlind1}
G\"unl\"uk, O., Linderoth, J.: Perspective reformulations of mixed integer
  nonlinear programs with indicator variables.
\newblock Mathematical Programming, Series B \textbf{124}, 183–--205 (2010)

\bibitem{perspecbook}
Hiriart-Urruty, J.B., Lemar\'{e}chal, C.: Convex analysis and minimization
  algorithms. {I}: Fundamentals, \emph{Grundlehren der Mathematischen
  Wissenschaften}, vol. 305.
\newblock Springer-Verlag, Berlin (1993)

\bibitem{KLS1997}
Ko, C.W., Lee, J., Steingr{\'{\i}}msson, E.: The volume of relaxed
  {B}oolean-quadric and cut polytopes.
\newblock Discrete Mathematics \textbf{163}(1-3), 293--298 (1997)

\bibitem{Lee2007}
Lee, J.: Mixed integer nonlinear programming: Some modeling and solution
  issues.
\newblock IBM Journal of Research and Development \textbf{51}(3/4), 489--497
  (2007)

\bibitem{LM1994}
Lee, J., Morris Jr., W.D.: Geometric comparison of combinatorial polytopes.
\newblock Discrete Applied Mathematics \textbf{55}(2), 163--182 (1994)

\bibitem{LeeSkipper2017}
Lee, J., Skipper, D.: Virtuous smoothing for global optimization.
\newblock Journal of Global Optimization \textbf{69}(3), 677--697 (2017)

\bibitem{LeeSkipperBQP2017}
Lee, J., Skipper, D.: Volume computation for sparse boolean quadric
  relaxations.
\newblock {\rm To appear in:} Discrete Applied Mathematics  (2017).
\newblock Doi:10.1016/j.dam.2018.10.038

\bibitem{cookbook}
{MOSEK ApS}: {MOSEK Modeling Cookbook, Release 3.1} (2019).
\newblock \url{https://docs.mosek.com/MOSEKModelingCookbook-letter.pdf}

\bibitem{SpeakAkerov2019}
Speakman, E., Averkov, G.: Computing the volume of the convex hull of the graph
  of a trilinear monomial using mixed volumes.
\newblock Discrete Applied Mathematics  (2019).
\newblock To appear

\bibitem{SpeakmanLee2015}
Speakman, E., Lee, J.: Quantifying double {M}c{C}ormick.
\newblock Mathematics of Operations Research \textbf{42}(4), 1230--1253 (2017)

\bibitem{SpeakmanLee_Branching}
Speakman, E., Lee, J.: On branching-point selection for trilinear monomials in
  spatial branch-and-bound: the hull relaxation.
\newblock Journal of Global Optimization \textbf{72}(2), 129--153 (2018)

\bibitem{SpeakmanYuLee}
Speakman, E., Yu, H., Lee, J.: Experimental validation of volume-based
  comparison for double-{M}c{C}ormick relaxations.
\newblock In: D.~Salvagnin, M.~Lombardi (eds.) CPAIOR 2017, pp. 229--243.
  Springer (2017)

\bibitem{SpeakmanThesis}
Speakman, E.E.: Volumetric guidance for handling triple products in spatial
  branch-and-bound.
\newblock {Ph.D.}, University of Michigan (2017)

\bibitem{Stein}
Steingr{\'\i}msson, E.: A decomposition of {$2$}-weak vertex-packing polytopes.
\newblock Discrete \& Computational Geometry. \textbf{12}(4), 465--479 (1994)

\bibitem{sdpt3}
Toh, K.C., Todd, M.J., T\"ut\"unc\"u, R.H.: {SDPT3} - a {MATLAB} software
  package for semidefinite programming.
\newblock Optimization Methods and Software \textbf{11}, 545--581 (1998)

\bibitem{XuLeeSkipper2019}
Xu, L., Lee, J., Skipper, D.: More virtuous smoothing.
\newblock SIAM Journal on Optimization \textbf{29}(2), 1240--1259 (2019)

\end{thebibliography}

\end{document}